\documentclass[a4paper, 11pt]{article}
\usepackage{amsmath, amsfonts, amsthm, latexsym}
\usepackage{graphicx}

\newtheorem{theorem}{Theorem}[section]

\newtheorem{proposition}[theorem]{Proposition}

\newtheorem{corollary}[theorem]{Corollary}

    \newcommand{\T}{{\mathbb T}}
 
\newcommand{\R}{{\mathbb R}}

\begin{document}
\title{An ergodic problem for Mean Field Games:\\qualitative properties and numerical simulations}
\author{Simone Cacace\footnotemark[1],\,\, Fabio Camilli\footnotemark[2],\,\, Annalisa Cesaroni\footnotemark[3]\,\, and Claudio Marchi\footnotemark[4] }
\date{}
\vskip -1.5cm

\maketitle

\footnotetext[1]
{
 Dip. di Matematica e Fisica, Universit{\`a}  degli Studi Roma Tre, Largo San Leonardo Murialdo 1, 00146 Roma, Italy, 
 ({\tt e-mail: cacace@mat.uniroma3.it})
}
\footnotetext[2]
{
 Dip. di Scienze di Base e Applicate per l'Ingegneria,  ``Sapienza" Universit{\`a}  di Roma, 
 via Scarpa 16, 00161 Roma, Italy, ({\tt e-mail:camilli@dmmm.uniroma1.it})
}
\footnotetext[3]
{
 Dip. di Scienze Statistiche, Universit\`a di Padova, via Cesare Battisti 241-243, 35121 Padova, Italy 
 ({\tt e-mail: annalisa.cesaroni@unipd.it})
}
\footnotetext[4]
{
 Dip. di Ingegneria dell'Informazione, Universit\`a di Padova, via Gradenigo 6/B, 35131 Padova, Italy 
 ({\tt e-mail: claudio.marchi@unipd.it})
}
\begin{abstract}
This paper is devoted to some qualitative descriptions and some numerical results for ergodic Mean Field Games systems which arise, e.g., in the homogenization with a small noise limit. We shall consider either power type potentials or logarithmic type ones. In both cases, we shall establish some qualitative properties of the effective Hamiltonian~$\bar H$ and of the effective drift~$\bar b$. In particular we shall provide two cases where the effective system keeps/looses the Mean Field Games structure, namely where $\nabla_P \bar H(P,\alpha)$ coincides or not with $\bar b(P, \alpha)$.

On the other hand, we shall provide some numerical tests validating the aforementioned qualitative properties of~$\bar H$ and~$\bar b$. In particular, we provide a numerical estimate of the discrepancy $\nabla_P \bar H(P,\alpha)-\bar b(P, \alpha)$.
\end{abstract}
 \begin{description}
\item [\textbf{ AMS subject classification}: ]  35B27, 35B30, 35B40, 35K40, 35K59, 65M06, 91A13 
  \item [\textbf{Keywords}: ] mean field games, periodic homogenization, small noise limit, ergodic problems, continuous dependence of solution on parameters, finite difference schemes
 \end{description}

\section{Introduction} 
This paper is devoted to some qualitative descriptions and some numerical results of ergodic problems arising in the homogenization of Mean Field Games (briefly, MFG) in the small noise limit. Our starting point is the homogenization result stated in~\cite{cdm}. In that paper, the authors tackle the asymptotic behaviour as $\epsilon\to 0$ of the solution to
 \begin{equation}\label{systemintro}\begin{cases}
-u_t^\epsilon-\epsilon\Delta u^\epsilon +\frac{1}{2}|\nabla u^\epsilon|^2=V\left(\frac{x}{\epsilon}, m^\epsilon \right), & x\in \mathbb{R}^n, t\in (0,T)\\
m^\epsilon_t-\epsilon \Delta m^\epsilon-{\rm div}(m^\epsilon\nabla u^\epsilon)=0,& x\in \mathbb{R}^n, t\in (0,T)
\end{cases}\end{equation}\rm
 with initial/terminal condition $u^0(x,T)=u_0(x)$ and $m^0(x,0)=m_0(x)$ where the potential $V(y,m)$ is $1$-periodic in $y$, increasing in $m$, $C^1$ and {\it bounded}.
It turns out that $(u^\epsilon,m^\epsilon)$ converge (in a suitable sense) to the solution of the effective Cauchy problem 
\begin{equation}\label{systemeffintro}\begin{cases}
-u_t^0+\bar H(\nabla u^0, m^0)=0, & x\in \mathbb{R}^n, t\in (0,T)\\
m^0_t-{\rm div}(m^0\bar b(\nabla u^0, m^0))=0,& x\in \mathbb{R}^n, t\in (0,T)
\end{cases}\end{equation}
 with the same initial/terminal condition. The effective operators $\bar H(P,\alpha)$ and $\bar b(P,\alpha)$ are obtained as follows: for every  $P\in \mathbb{R}^n$ and $\alpha\geq 0$, the value $\bar H(P,\alpha)$ is the (unique) constant for which there exists a solution to the ergodic MFG system:  
\begin{equation}\label{ucorintro} \begin{cases}
(i)   -\Delta u     \!+\!\frac{1}{2}|\nabla u  \!+\!P|^2\!-\!V(y, \alpha m  ) =
 \bar  H(P,\alpha), & y\in\mathbb{T}^n\\
(ii)   -\Delta m  -{\rm div}\left(m \left(\nabla u +P\right)\right)=0, & y\in\mathbb{T}^n\\ (iii)  \int_{\mathbb{T}^n} u =0 \qquad  \int_{\mathbb{T}^n} m =1, & \end{cases}
\end{equation} 
while $\bar b$ is given by  \begin{equation}\label{bbar} \bar b(P,\alpha)= \int_{\mathbb{T}^n} (\nabla u +P) m dy,\end{equation} where $(u,m)$ is the  solution to \eqref{ucorintro}.
Moreover the following relation holds true
\begin{equation}\label{H_P} 
\frac{\partial H}{\partial P_i}(P,\alpha) = \bar b_i(P,\alpha)-\alpha \int_{\T^n}V_m(y,\alpha m) \tilde m_i m dy
\end{equation}
where $V_m=\partial V/\partial m$ and $\left(\tilde u_i,\tilde m_i, \frac{\partial H}{\partial P_i}(P,\alpha) \right)\in C^{2,\gamma}\times W^{1,p}\times \R$ is the solution of the ergodic problem 
\begin{equation}\label{mvarP}
\begin{cases} (i) &  -\Delta\tilde u_i + \nabla\tilde u_i\cdot (\nabla u+P) + (\nabla u+P)\cdot e_i -V_m(y,\alpha m)\alpha\tilde m_i=\frac{\partial H}{\partial P_i}(P,\alpha) \\ (ii)&
-\Delta \tilde m_i-{\rm div}\big((P+\nabla u)\tilde m_i \big) ={\rm div }( m(\nabla \tilde u_i+ e_i)) \\(iii) & \int_{\mathbb{T}^n}\tilde m_i=\int_{\mathbb{T}^n}\tilde u_i=0. \end{cases}
\end{equation} 

It is worth to observe that, for the homogenization of a MFG system with finite noise, the cell problem is {\it decoupled}, see \cite[Sect.3.1.2]{cdm}; hence, from a mathematical point of view, the most interesting case is the small noise one as in \eqref{systemintro}.

Let us recall that the ergodic system \eqref{ucorintro} is not related only to the homogenization problem; actually, it has its own interpretation as the MFG system when agents pay an infinite horizon cost (see \cite{be, ll1, ll}) and it arises also in the study of the asymptotic behaviour as $t\to+\infty$ of evolutive MFG system (see \cite{cllp1, cllp2, cp}). For a general overview on MFG systems and on their applications, we refer the reader to the monographs \cite{noteachdou, bf, notecard, clld, cc, gn,gbook,notegomes} and references therein.

The aim of this paper is to study the ergodic systems \eqref{ucorintro} and \eqref{mvarP} under more general assumptions, especially we shall consider {\it unbounded} potential of the form
\begin{enumerate}
\item Power type nonlinearities: $V(y,m)=v(y)+ m^q$ for some $q>0$;
\item Logarithmic nonlinearities: $V(y,m)= v(y)+\log m$ \end{enumerate} 
where $v:\T^n\to \R$ is Lipschitz continuous and $1$-periodic. Without any loss of generality, we shall assume that $v\geq 0$. 
Let us stress that our methods apply to regular solutions to \eqref{ucorintro},  that is $(u,m)\in C^{2,\gamma}(\T^n)\times W^{1,p}(\T^n)$, for all $\gamma\in (0,1)$ and for all $p>1$. As a matter of fact, by the regularity of the Hamiltonian, under our assumptions, the solution to \eqref{ucorintro} belongs to $C^{2,\gamma}(\T^n)\times C^{2,\gamma}(\T^n)$, for all $\gamma\in (0,1)$.

An interesting feature is that the limit system \eqref{systemeffintro} may loose the MFG structure, namely $\nabla_P \bar H(P,\alpha)$ may not coincide 
with $\bar b(P, \alpha)$. This feature was already pointed out in \cite[Sect.6.1]{cdm} for the potential $V(y,m)=v(y)+m$; in Corollary \ref{coro1} we shall extend 
it to more general cases of power type nonlinearities. On the other hand, in Theorem \ref{thmderivatives11} we shall show $\nabla_P \bar H(P,\alpha)=\bar b(P, \alpha)$  
for MFG systems with logarithmic potentials. We will finally explore these results numerically, in particular by computing the discrepancy $\nabla_P \bar H- \bar b$. 

\section{Mean field games with power nonlinearities}\label{sect:power}
In this section we consider potential of the form $V(y,m)=v(y)+ m^q$ with $q>0$. 
In this case, the ergodic problem \eqref{ucorintro} reads 
\begin{equation}\label{s2}
 \begin{cases}
(i) -\Delta u +\frac{|\nabla u+P|^2}{2} - v(y)- \alpha^q m^q=\bar H(P, \alpha), & y\in\mathbb{T}^n\\
(ii) -\Delta m-{\rm div}  ( m (\nabla u+P))=0, & y\in\mathbb{T}^n\\ (iii)  \int_{\mathbb{T}^n} u =0 \qquad  \int_{\mathbb{T}^n} m =1. & \end{cases}
\end{equation}  
Throughout this section we shall assume
\begin{equation*}q>0\quad\textrm{if }n\leq 4,\qquad 0<q\leq \frac{2}{n-2}\quad \textrm{if } n>4.
\end{equation*}
We collect some qualitative properties of the effective operators. 
\begin{proposition}\label{prop1} There hold 
\begin{itemize}
\item[(i)] For every $P\in\R^n$, $\alpha\geq 0$ there exists a unique constant $\bar H(P,\alpha)$ such that \eqref{s2} admits a solution $(u,m)$. Moreover this solution is unique, 
$(u,m)\in C^{2,\gamma}(\T^n)\times W^{1,p}(\T^n)$ for every $\gamma\in (0,1)$ and $p>1$ and $m>0$.
\item[(ii)] $\bar H(P,\alpha)$  is   decreasing in $\alpha$.
\item[(iii)]$\bar H(P,\alpha)$ is coercive in $P$, with quadratic growth, in particular there exists a constant $R_q\geq 1$ depending only on $q$ such that 
\[ \frac{|P|^2}{2}-(1+R_q)\int_{\T^n} v(y)dy-R_q\alpha^q\leq   \bar H(P,\alpha)\leq  \frac{|P|^2}{2 }-\alpha^q.\]

\item[(iv)] There hold
\begin{equation*}
\lim_{|P|\to +\infty}\frac{\bar H(P, \alpha)}{|P|^2} =\frac{1}{2}\qquad \text{and}\qquad \lim_{|P|\to +\infty}\frac{|\bar b(P, \alpha)-P |}{|P|}=0,\end{equation*} 
locally uniformly for $\alpha \in [0,+\infty)$.\end{itemize}
\end{proposition} 
\begin{proof}
(i) This existence result can be found in \cite[Thm 7.1]{gomespatrizi}, \cite[Theorem 1.4]{cirant}, see also \cite{pv} and \cite{cc}. 
\\(ii) Following the same argument as in  \cite[Proposition 3]{cdm}
we get that for $\alpha_1, \alpha_2\geq 0$, 
\[(\bar H(P,\alpha_1)-\bar H(P,\alpha_2))(\alpha_1-\alpha_2)=-\int_{\T^n} (m_1^q-m_2^q)(m_1-m_2)dy
\]
where $m_1$ and $m_2$ are respectively the solution to \eqref{s2} with $\alpha=\alpha_1$ and $\alpha=\alpha_2$. 
Since $q>0$, we get the statement.
\\(iii)  We multiply equation (i) in \eqref{s2} by $m$, we integrate and we get, recalling the periodicity assumptions  and that $m$ has mean $1$, 
\begin{equation*}
\bar H(P,\alpha)=\int_{\T^n} -m\Delta u +\frac{1}{2 }\left|\nabla u +P\right|^2 m-\alpha^q m^{q+1} -v(y)mdy. \end{equation*}
We multiply equation (ii) in \eqref{s2}  by $u$, integrate and subtract by the previous one to get
\begin{equation}\label{c2dim1}\bar H(P,\alpha)=\int_{\T^n} \frac{1}{2 }(|P|^2-|\nabla u|^2)m -\alpha^q m^{q+1} -v(y)mdy.  \end{equation}
Therefore, recalling that $v\geq 0$, and that by Jensen inequality $\int_{\T^n} m^q dy\geq(\int_{\T^n} mdy)^q=1$,  we get 
$\bar H(P,\alpha)\leq  \frac{|P|^2}{2 }-\alpha^q $. 
\\ Integrating the first equation in \eqref{s2}, we obtain 
\begin{eqnarray}\nonumber\bar H(P,\alpha)&=& \int_{\T^n}  \frac{|\nabla u+P|^2}{2} - v(y)- \alpha^q m^q dy
\\ \label{c3dim1} &=& \frac{|P|^2}{2}+\int_{\T^n}  \frac{|\nabla u|^2}{2}- v(y)- \alpha^q m^q dy.
\end{eqnarray}
From this, using  \eqref{c2dim1}, recalling that $m$ has mean $1$, we get 
\begin{equation}\label{e1dim1}\int_{\T^n} \frac{1}{2 } |\nabla u|^2 (m+1)+(m-1) (\alpha^q m^{q} +v(y))dy=0.  \end{equation}
By monotonicity of $m\mapsto \alpha^qm^q$ on $m\geq 0$, we observe that \begin{equation}\label{mon} \alpha^q m^q(m -1) \geq \alpha^q (m  -1)\end{equation}
and then we conclude from \eqref{e1dim1} that 
\begin{equation}\label{e2dim1}\int_{\T^n} \frac{1}{2 } |\nabla u|^2 (m+1)dy\leq - \int_{\T^n} (m-1) v(y)dy\leq  \int_{\T^n}  v(y)dy.  \end{equation}
Integrating \eqref{mon}, we get 
$\int_{\T^n} m^qdy\leq \int_{\T^n} m^{q+1}dy.$ Hence, in \eqref{e1dim1}, we obtain 
\[ 0\leq \int_{\T^n} \alpha^q m^{q+1}-\alpha^q m^q dy\leq    \int_{\T^n} v(y)(1-m) dy\leq \int_{\T^n} v(y)dy. \]
Therefore, by applying H\"older inequality, we obtain
\[\alpha^q \int_{\T^n} m^{q} dy\leq  \alpha^q \int_{\T^n} m^{q+1} dy \leq  \alpha^q  \left(\int_{\T^n} m^{q+1} dy \right)^{\frac{q}{q+1}}+  
 \int_{\T^n}  v(y) dy. \]
Let  us observe that this inequality can be written as $A\leq B+ A^{\frac{q}{q+1}}$. So, if we choose $R_q>1$   such that $R_q\geq 1+ R_q^{\frac{q}{q+1}}$, 
then the solutions to the inequality satisfies $A\leq R_q(B+1)$. 
Hence, we conclude that \begin{equation}\label{lq+1}\alpha^q  \int_{\T^n}  m^{q}\leq \alpha^q\int_{\T^n}  m^{q+1} \leq R_q\left( \int_{\T^n} v(y) dy +\alpha^q\right).\end{equation}
Replacing this inequality in \eqref{c3dim1}, we obtain 
\[ \bar H(P,\alpha) 
\geq \frac{|P|^2}{2}-(1+R_q)\int_{\T^n} v(y)dy-R_q\alpha^q.\] 
\\(iv)  The first limit is a direct consequence of item $(iii)$. 
By  definition \eqref{bbar}, we get \[\bar b(P, \alpha)-P =\int_{\T^n}   \nabla u(y) m dy.\] 
 By H\"older's  inequality and \eqref{e2dim1} we get
\[\int_{\T^n}  |\nabla u(y)| m dy\leq  \left(\int_{\T^n}  |\nabla u |^2m dy \right)^{\frac{1}{2}} \left(\int_{\T^n}   m dy \right)^{\frac{1}{2}}\leq \left(\int_{\T^n} v(y) dy \right)^{\frac{1}{2}} .\]
This permits to conclude. 
\end{proof}

Using the properties stated in Proposition \ref{prop1}, we obtain regularity properties of the effective operators, and a relation between them. 
\begin{theorem}\label{thmderivatives1} 
The maps  $(P,\alpha)\to \bar H(P, \alpha), \bar b(P,\alpha)$ are locally Lipschitz continuous and admit partial derivatives everywhere. 
Moreover, there hold
\begin{eqnarray}\label{nablaP}\nabla_P \bar H(P,\alpha)&=&  \bar b(P,\alpha) -\frac{q}{q+1} \alpha^q \nabla_P \| m\|_{L^{q+1}}^{q+1}
\\ \frac{\partial \bar b}{\partial P_i}\cdot e_i&\geq& 0. \label{bbarrader} \end{eqnarray}
\end{theorem} 
\begin{proof}
The proof of the local Lipschitz continuity of $\bar H$ follows the same arguments of \cite[Proposition 5]{cdm} by using the a priori bounds on $m,\nabla u$ obtained in the previous proof, in particular \eqref{e2dim1} and \eqref{lq+1}. 

We fix $P\in\R^n$, $\alpha>0$; let $(u,m)$ be  the solution to \eqref{s2} associated to $(P,\alpha)$.  
We introduce the following system for all $i=1,\dots, n$ 
\begin{equation}\label{s2der}
\begin{cases} (i) &  -\Delta\tilde u_i + \nabla\tilde u_i\cdot (\nabla u+P) + (\nabla u+P)\cdot e_i -q\alpha^q  m^{q-1}\tilde m_i=c_i(P,\alpha) \\ (ii)&
-\Delta \tilde m_i-{\rm div}\big((P+\nabla u)\tilde m_i \big) ={\rm div }(m(\nabla \tilde u_i+ e_i)) \\
(iii)& \int_{\mathbb{T}^n}\tilde m_i=\int_{\mathbb{T}^n}\tilde u_i=0. \end{cases}
\end{equation} 
Arguing as in \cite[Lemma 6]{cdm}, we get that for every $i$ there exists a unique $c_i(P,\alpha)$ such that this system admits a solution $(\tilde u_i,\tilde m_i)$,
which is regular and unique. Moreover, the compatibility condition for existence of solutions reads
\[c_i(P,\alpha)= \int_{\T^n} (\nabla u+P)_i m- \alpha^q q m^q \tilde m_i dy=\bar b_i(P,\alpha)-q\alpha^q \int_{\T^n} m^q \tilde m_i dy\]
where the latter equality is due to \eqref{bbar}.

We fix now $\delta\in \R$ and define $u_i^\delta, m_i^\delta$ the solution to \eqref{s2} associated to $(P+\delta e_i, \alpha)$. Then it is possible to prove arguing as in \cite[Theorem 7]{cdm} that $\frac{u_i^\delta-u}{\delta},\frac{m_i^\delta-m}{\delta}$ converges as $\delta\to 0$ to a solution to \eqref{s2der} in a suitable sense (weakly in $W^{1,2}$ for $\frac{u_i^\delta-u}{\delta}$ and weakly in $L^2$ for $\frac{m_i^\delta-m}{\delta}$).
This implies by uniqueness that, up to subsequences,  $c_i(P,\alpha)=\lim_{\delta\to 0} \frac{\bar H(P+\delta e_i, \alpha)-\bar H(P,\alpha)}{\delta}$. So relation \eqref{nablaP} is exactly the compatibility condition for existence of solutions to \eqref{s2der}. 
\\ 
Following the same arguments as in \cite[Theorem 11]{cdm}, we get that $\bar b$ is locally Lipschitz continuous and moreover 
\[\frac{\partial \bar b}{\partial P_i}=e_i +\int_{\T^n}\tilde m_i \nabla u + m\nabla \tilde u_i dy.  \]
 \\
We multiply (i) and (ii) in \eqref{s2der}  respectively by $\tilde m_i$ and by $\tilde u_i$, we subtract and we get
\[\int_{\T^n}  \tilde m_i \nabla u \cdot e_i   -q\alpha^q  m^{q-1}\tilde m_i^2 - m(\nabla \tilde u_i +e_i)\cdot \nabla \tilde u_i dy=0.\]
Therefore  
\begin{eqnarray*} \frac{\partial \bar b}{\partial P_i}\cdot e_i&=& 1+\int_{\T^n} \tilde m_i \nabla u\cdot e_i + m\nabla \tilde u_i \cdot e_idy\\
&=&1+\int_{\T^n} q\alpha^q  m^{q-1}\tilde m_i^2 +m(\nabla \tilde u_i +e_i)\cdot \nabla \tilde u_i+ m\nabla \tilde u_i \cdot e_idy\\
&=&1+ \int_{\T^n} q\alpha^q  m^{q-1}\tilde m_i^2 +m|\nabla \tilde u_i +e_i|^2- m dy\\
  &=& \int_{\T^n} q\alpha^q  m^{q-1}\tilde m_i^2 +m|\nabla \tilde u_i +e_i|^2dy\end{eqnarray*}
and relation \eqref{bbarrader} easily follows.
\end{proof} 
\begin{corollary} \label{coro1} 
Under the assumption of Theorem \ref{thmderivatives1}, for $n=1$, there holds
\[\nabla_P \bar H(P,\alpha)\neq \bar b (P,\alpha).\] 
\end{corollary} 
\begin{proof}  First of all we observe that by Jensen inequality $\| m\|_{L^{q+1}}^{q+1}>1$. 

Moreover, arguing as in \cite[Proposition 3.4]{gomespatrizi} or as in \cite[Proposition 14]{cdm} (see also \cite[Lemma 2.5]{cllp1}) and still using \eqref{e2dim1}, we have that $\sqrt{m}$ is bounded in $W^{1,2}(\T^n)$, uniformly in $P$. In other words, there exists a constant $C>0$ independent of $P$ such that
\[\int_{\T^n} |\nabla \sqrt{m}|^2 dy
\leq C\int_{\T^n} v(y)dy.\]
So, since $n=1$, we conclude that $m\to \bar m$ strongly in $L^{q+1}(\T^1)$. On the other side, reasoning as in \cite[Proposition 14]{cdm} 
we can prove that $m$ converges weakly to $1$ in $L^{q+1}$. This implies that  $\| m\|_{L^{q+1}}^{q+1}\to 1$ as $|P|\to +\infty$ and then in particular 
$\nabla_P \bar H(P,\alpha)\neq \bar b (P,\alpha)$ at some $(P,\alpha)$. 
\end{proof} 
\section{Mean field game with logarithmic  nonlinearities}\label{sect:log}
In this section we consider potential of the form $V(y,m)=v(y)+\log m$. 
In this case, defining $\bar H(P)$ as \[\bar H(P, \alpha)=\bar H(P)+\log \alpha\qquad \forall \alpha>0\] the system \eqref{ucorintro} becomes
\begin{equation}\label{s3}
 \begin{cases}
(i) -\Delta u + \frac{|\nabla u+P|^2}{2} - v(y)-\log m=\bar H(P), & y\in\mathbb{T}^n\\
(ii) -\Delta m-{\rm div} ( m (\nabla u+P))=0, & y\in\mathbb{T}^n\\ (iii)  \int_{\mathbb{T}^n} u =0 \qquad  \int_{\mathbb{T}^n} m=1. & \end{cases}
\end{equation} 

We  collect some qualitative properties of the effective operators. 
\begin{proposition}\label{prop11} \ \  \ \begin{itemize}
\item[(i)] For every $P\in\R^n$,  there exists a unique constant $\bar H(P)$ such that \eqref{s3} admits a solution $(u,m)$. Moreover this solution is unique, 
$(u,m)\in C^{2,\gamma}(\T^n)\times W^{1,p}(\T^n)$ for every $\gamma\in (0,1)$ and $p>1$ and $m>0$.
\item[(ii)]$\bar H(P)$ is coercive in $P$, with quadratic growth, in particular  there exists $C>0$ such that
\[ \frac{|P|^2}{2}-C-2\int_{\T^n}  v(y)dy \leq \bar H(P)\leq \frac{|P|^2}{2}.\]
\item[(iii)] $\bar H(P)$ is strictly convex. 
\item[(iv)] There hold \begin{equation*}
  \lim_{|P|\to +\infty}\frac{\bar H(P)}{|P|^2} =\frac{1}{2}\qquad \text{and}\qquad 
\lim_{|P|\to +\infty}\frac{|\bar b(P)-P |}{|P|}=0.\end{equation*} 
 \end{itemize}
\end{proposition} 
\begin{proof}
 Following the arguments in \cite{gsm}, we introduce the following minimization problem. For any $\phi$ smooth and periodic we define the energy 
 \[ E_P(\phi)= \int_{\T^n} e^{-\Delta \phi+  \frac{|\nabla \phi+P|^2}{2} - v(y)} dy\] and consider the minimization problem
 \[\inf \{E_P(\phi)\ |\ \text{ $\phi$ smooth and periodic}\}.\] 
 Note that $\inf E_P(\phi)\leq  E_P(0)\leq e^{\frac{|P|^2}{2}}$. 
According to \cite[Theorem 1]{gsm}, there exists a unique (up to constants) smooth minimizer  $\phi$   for this problem. We define 
\[m_\phi(y):= \frac{e^{-\Delta \phi+ \frac{|\nabla \phi+P|^2}{2} - v(y)}}{E_P(\phi)}.\]
\\
(i) By computing the Euler Lagrange equation associated to the minimization problem we get that   $u=\phi$, $m=m_\phi$ are solution to \eqref{s3} with $\bar H$ given by  
\begin{equation}\label{rep} \bar H(P)=\log \inf\{ E_P(\phi)\ |\ \text{$\phi$ smooth and periodic}\}.\end{equation}  So, \cite[Theorem 1]{gsm} gives the existence  and uniqueness of a regular solution to \eqref{s3} (see also \cite[Thm 7.1]{gomespatrizi}). 
\\(ii) Using the representation formula \eqref{rep} for $\bar H(P)$, we get that  $\bar H(P)\leq \log  E_P(0)\leq \frac{|P|^2}{2}$. \\
 We multiply the first equation in \eqref{s3} by $m$, the second by $u$, integrate and subtract, and we obtain 
\begin{equation*}
\bar H(P)=\int_{\T^n} \frac{1}{2 }(|P|^2-|\nabla u|^2)m  -m\log m -v(y)mdy.  \end{equation*}
 Integrating the first equation in \eqref{s3}, 
we obtain 
\begin{multline}\label{c3dim2}\bar H(P)= \int_{\T^n}  \frac{|\nabla u+P|^2}{2}- \log m  - v(y)dy=
 \frac{|P|^2}{2}+\int_{\T^n}  \frac{|\nabla u|^2}{2}-  \log m - v(y) dy.
\end{multline}
The last two equalities give
 \begin{equation}\label{e1dim2}\int_{\T^n} \frac{1}{2 } |\nabla u|^2 (m+1)+(m-1) v(y)+ (m\log m-\log m)dy=0.  \end{equation}
Note that $m\log m\geq \log m$, so we conclude that also in this case  \eqref{e2dim1} holds.  Moreover, we observe that there exists $C>0$ such that $m\log m\geq 2\log m-C$ and then
\[\int_{\T^n} (m-1) v(y)+ \log m dy-C\leq \int_{\T^n} (m-1) v(y)+ (m\log m-\log m)dy\leq 0.\]
This in turns gives 
\[\int_{\T^n} \log m dy \leq C+\int_{\T^n}  v(y)dy.\]
So, again by \eqref{c3dim2}, we conclude $\bar H(P)\geq \frac{|P|^2}{2}-C-2\int_{\T^n}  v(y)dy.$
\\ (iii) The strict convexity is  a consequence of the representation formula \eqref{rep}, see \cite[Lemma1]{gsm}. 
 \\ (iv) The proof is the same as in Proposition \ref{prop1}. 
\end{proof}
As before, using the properties stated in Proposition \ref{prop11}, we obtain regularity properties of the effective operators, and a relation between them. 
\begin{theorem}\label{thmderivatives11} 
The maps  $P\to \bar H(P), \bar b(P )$ are locally Lipschitz continuous and admit partial derivatives everywhere. 
Moreover, there hold
\begin{equation}\label{nablaP1}\nabla_P \bar H(P)=  \bar b(P). \end{equation}
\end{theorem} 
\begin{proof} The regularity comes from the properties of $\bar H$ proved in Proposition \ref{prop11}.
\\ We fix $P\in\R^n$; let $(u,m)$ be the solution to \eqref{s3} associated to $P$.  
We introduce the following system for all $i=1,\dots, n$ 
\begin{equation}\label{s2der1}
\begin{cases} (i) &  -\Delta\tilde u_i + \nabla\tilde u_i\cdot (\nabla u+P) + (\nabla u+P)\cdot e_i -\frac{\tilde m_i}{m}=c_i(P) \\ (ii)&
-\Delta \tilde m_i-{\rm div}\big((P+\nabla u)\tilde m_i \big) ={\rm div }(m(\nabla \tilde u_i+ e_i)) \\
(iii)& \int_{\mathbb{T}^n}\tilde m_i=\int_{\mathbb{T}^n}\tilde u_i=0. \end{cases}
\end{equation} Note that by \cite[Theorem 5.1]{gomespatrizi} $\frac{1}{m}\in L^\infty(\T^n)$.
Arguing as in \cite[Lemma 6]{cdm}, we get that for every $i$ there exists a unique $c_i(P)$ for which this system admits a solution $(\tilde u_i,\tilde m_i)$,
which is regular and unique. Moreover, the compatibility condition for existence of solutions reads
\[c_i(P)= \int_{\T^n} (\nabla u+P)_i m- \tilde m_i dy =\int_{\T^n} (\nabla u+P)_i m dy= \bar b_i(P,\alpha).\]
Moreover as in Theorem \ref{thmderivatives11}, we get that, up to subsequences,  $c_i(P)=\lim_{\delta\to 0} \frac{\bar H(P+\delta e_i)-\bar H(P)}{\delta}$. So relation \eqref{nablaP1} is exactly the compatibility condition for existence of solutions to \eqref{s2der1}. 
\end{proof} 

\section{Numerical Tests}\label{sect:numer}
This section is devoted to the numerical solution of the cell problems presented in the previous sections, 
in particular we provide a numerical validation of the properties involving the effective Hamiltonian $\bar H$ and the effective drift $\bar b$. 
We recall that the computation of effective operators is an expensive task, since it requires the solution of a nonlinear problem for each pair $(P,\alpha)$.
To this end we apply the {\em direct method} we recently introduced in  \cite{cc16,ccm17}, 
an efficient solver for ergodic problems associated to the homogenization of Hamilton-Jacobi equations. 
The new method avoids completely classical approximations such as the \emph{small-delta} or the \emph{large-time} methods (see \cite{acd, q}). 
The ergodic constant $\bar H$ appearing in the cell problem \eqref{ucorintro} is handled as an additional unknown, so that the corresponding nonlinear system is 
formally \emph{overdetermined}, having more equations than unknowns. 
The solution is then meant in a \emph{least-squares} sense and computed by a generalized Newton method for inconsistent systems. 
We refer the reader to \cite{cc16,ccm17} for technical details and for the actual implementation of the method. Here we just remark that, 
after introducing a discretization of the torus $\T^n$ with $N$ nodes, collecting the approximations $U_i, M_i, \bar H$ ($i=1,...,N$) of the unknowns 
in a single vector $X=(U,M,\bar H)$ of length $2N+1$, and choosing a scheme 
which is consistent with the notion of viscosity solutions, we end up with a nonlinear system of $2N+2$ equations. 
The first $2N$ equations correspond to the partial differential equations in \eqref{ucorintro}, whereas the last $2$ equations correspond to the normalization conditions 
$\int_{\T^n} u=0$ and $\int_{\T^n}m=1$. Defining a suitable  non linear map $F:\R^{2N+1}\to\R^{2N+2}$, the resolution of the cell problem is then reduced to 
$$\mbox{Find $X\in\R^{2N+1}$ such that $F(X)=0\in\R^{2N+2}$}\,.$$
\noindent 
Note that, once a solution $X$ is computed, we obtain the effective drift $\bar b$ in \eqref{bbar} by a simple trapezoidal quadrature rule. 
Moreover,  the system \eqref{mvarP} can be discretized  in a similar  way and its solution easily computed in a single iteration, being an overdetermined but linear system. 
Hence we get the solution  $(\tilde U,\tilde M,\frac{\partial \bar H}{\partial P})$, from which we can readily compute the integral in \eqref{H_P}, 
again by a trapezoidal quadrature rule. This gives, in particular, a practical strategy to validate numerically the relationship described by  \eqref{H_P}.

Here we consider for simplicity the one dimensional case $n=1$, but the method can be easily implemented in any dimension. 
In all the tests we choose 
a nonnegative potential of the form $v(x)=A(1+\frac{1}{2}(\sin(2\pi x)+\cos(4\pi x)))$ with $A>0$. 
We observe that, for small values of the amplitude $A$, the corresponding {\em correctors} $u$ and $m$, i.e. solutions to \eqref{ucorintro}, are quite close to $0$ and $1$ 
respectively, as shown in Figure \ref{PotentialA} in the case $V(y,m)=v(y)+m$ with $P=10$ and $\alpha=2$. 
This considerably affects the shape of the effective Hamiltonian and effective drift, 
making difficult to explore numerically the features we are interested in. That's why we enforce non trivial solutions by choosing $A=100$ in what follows.  

\begin{figure}[h!]
 \begin{center}
 \begin{tabular}{ccc}
\includegraphics[width=.3\textwidth]{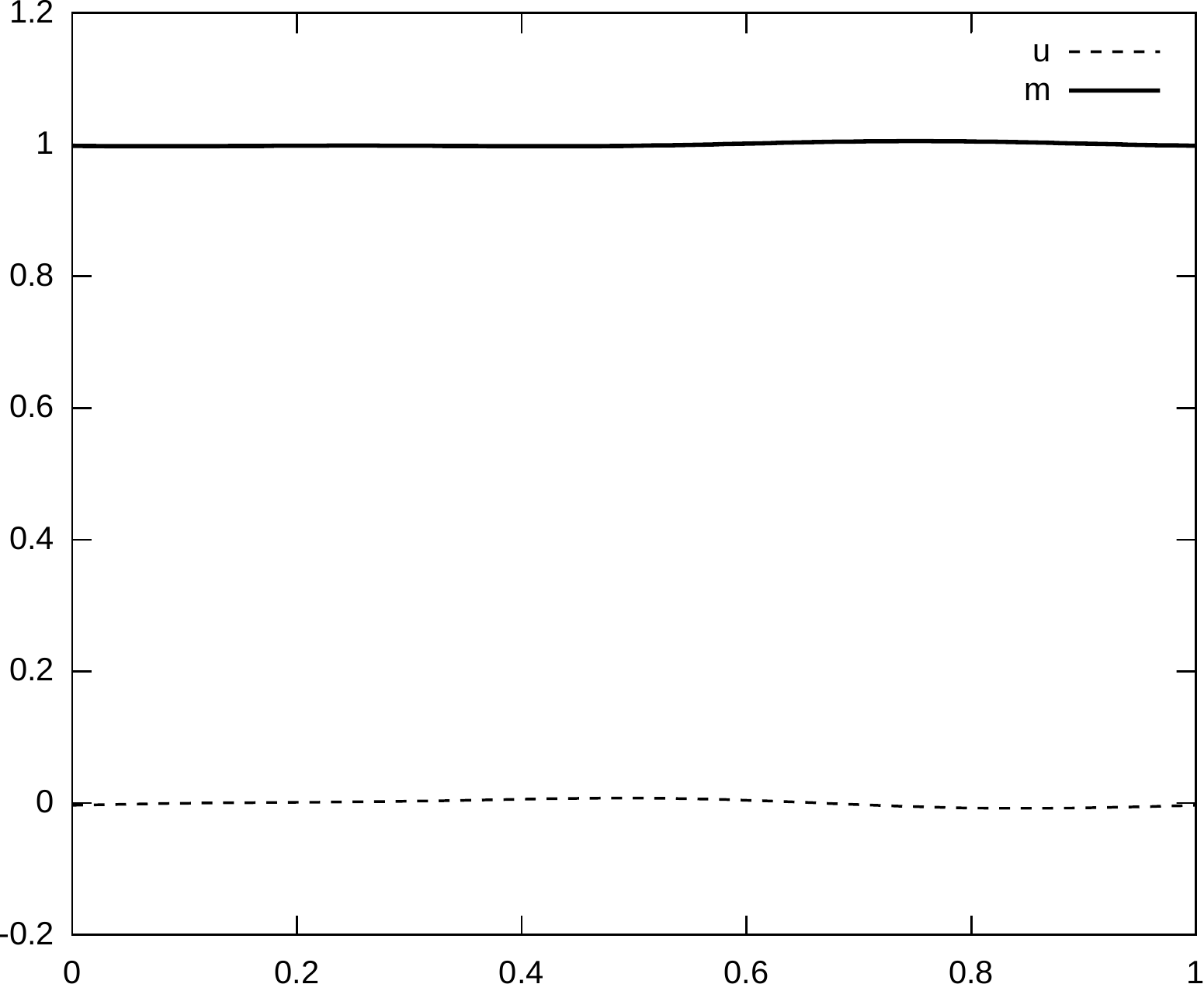} &
\includegraphics[width=.3\textwidth]{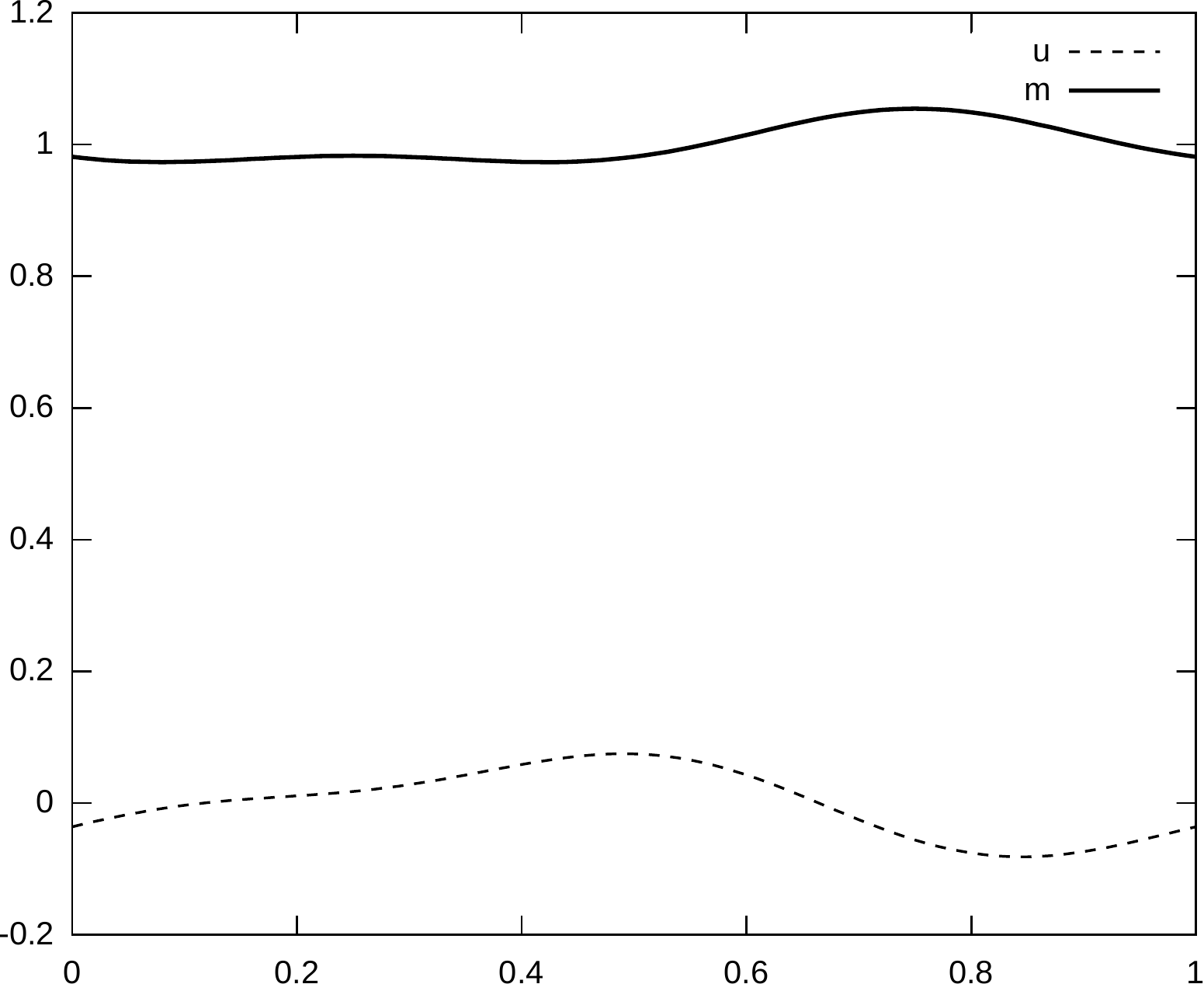} &
\includegraphics[width=.3\textwidth]{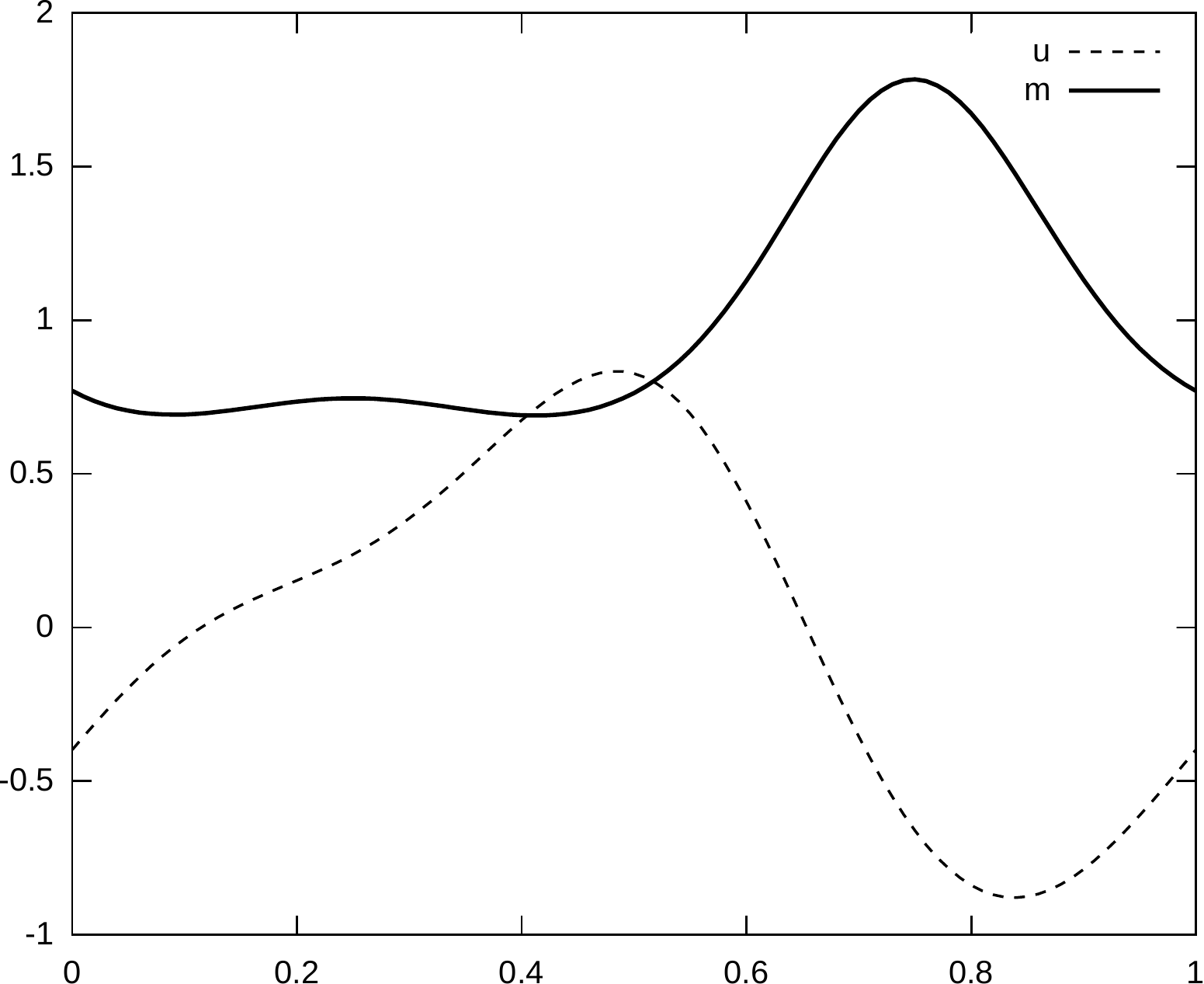} \\
(a)&(b)&(c)
\end{tabular}
\end{center}
\caption{Correctors of the cell problem \eqref{ucorintro} depending on the amplitude $A$ of the potential $v$: $A=1$ (a), $A=10$ (b), $A=100$ (c).}\label{PotentialA}
\end{figure}

We start with the case of power nonlinearities, setting $V(y,m)=v(y)+m^q$, with $q=1,2$. We discretize the torus $\T^1$ with $N=400$ 
nodes and we consider a uniform grid of $51\times 51$ nodes for discretizing the space of parameters $[-10,10]\times[0,20]\ni (P,\alpha)$. \\
Figure \ref{HBQ1} 
shows the surfaces and the level sets of the computed effective Hamiltonian $\bar H(P,\alpha)$ and the effective drift $\bar b(P,\alpha)$ as functions of $(P,\alpha)$ for $q=1$, 
whereas Figure \ref{HBQ2} shows the case $q=2$. 
\begin{figure}[h!]
 \begin{center}
 \begin{tabular}{cc}
\includegraphics[width=.45\textwidth]{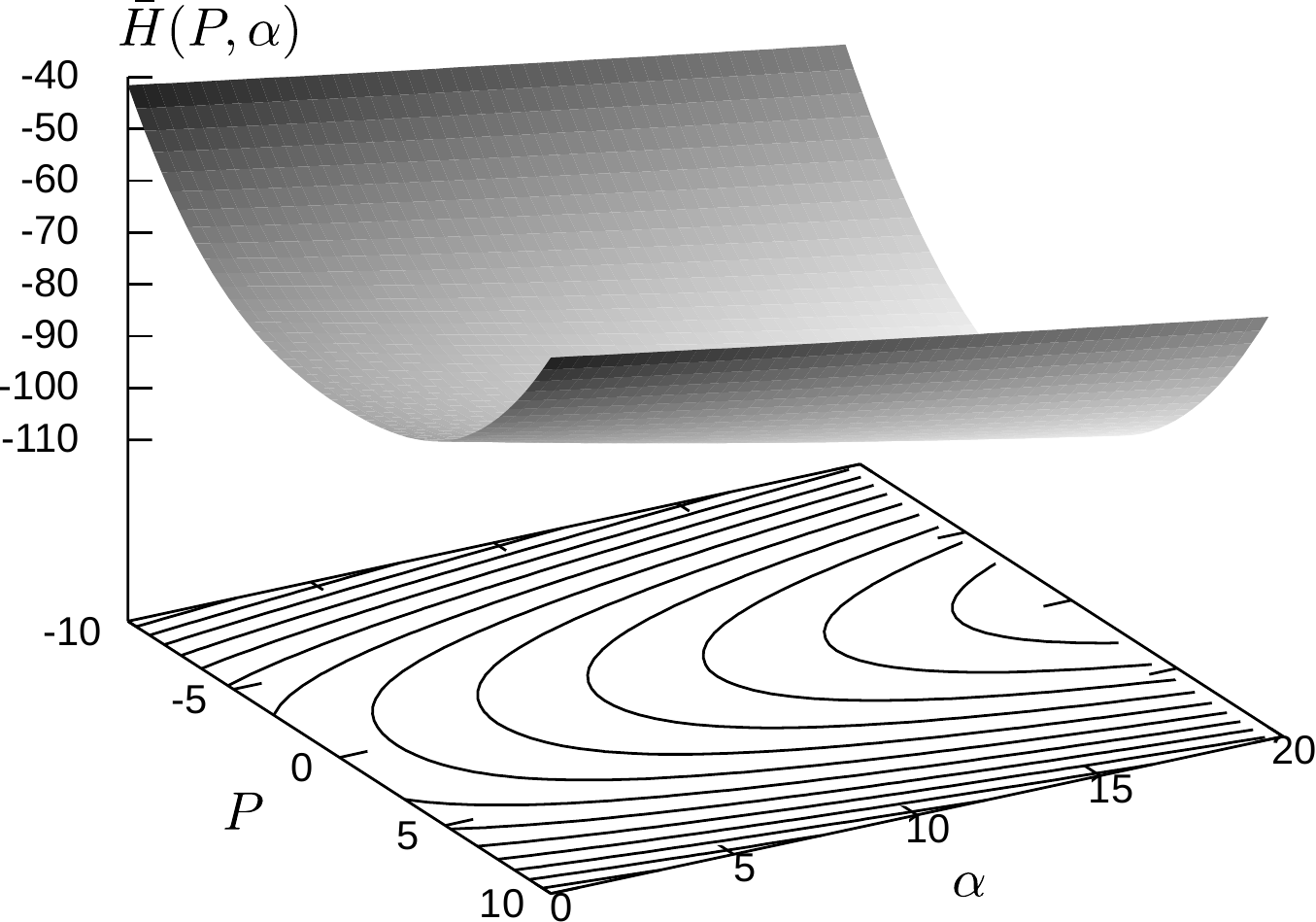} &
\includegraphics[width=.45\textwidth]{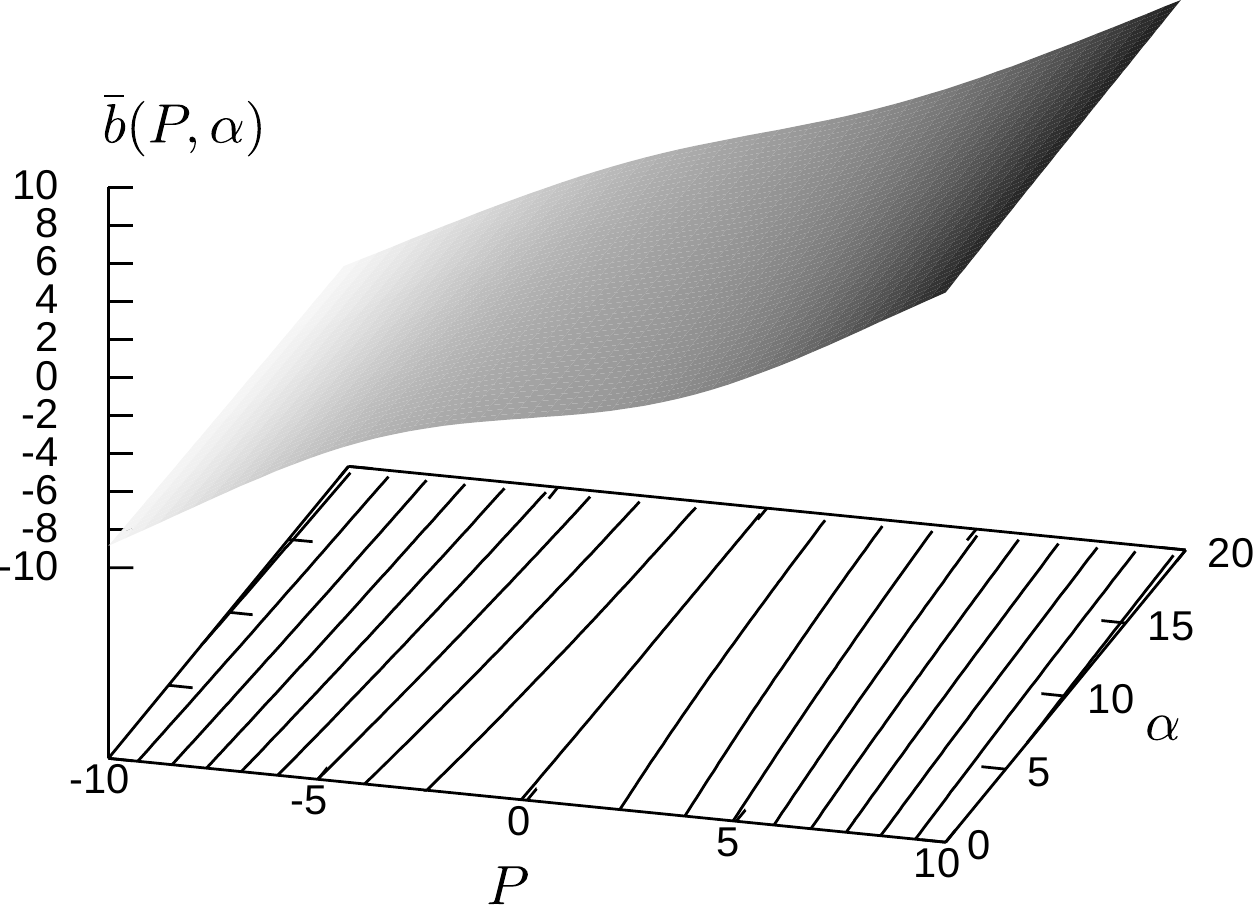} \\
(a)&(b)
\end{tabular}
\end{center}
\caption{Surfaces and level sets of $\bar H$ (a) and $\bar b$ (b) for $q=1$.}\label{HBQ1}
\end{figure}
\begin{figure}[h!]
 \begin{center}
 \begin{tabular}{cc}
\includegraphics[width=.45\textwidth]{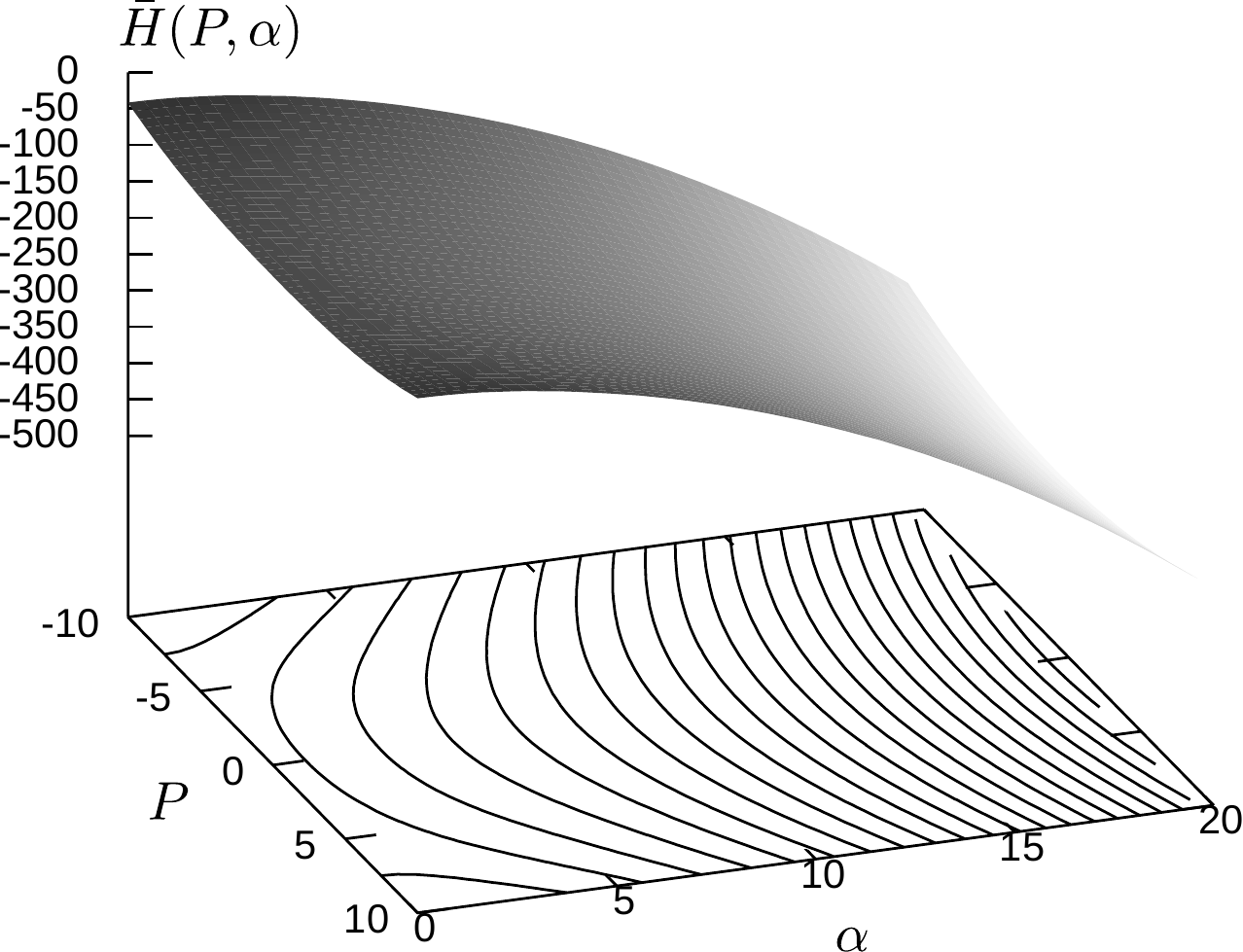} &
\includegraphics[width=.45\textwidth]{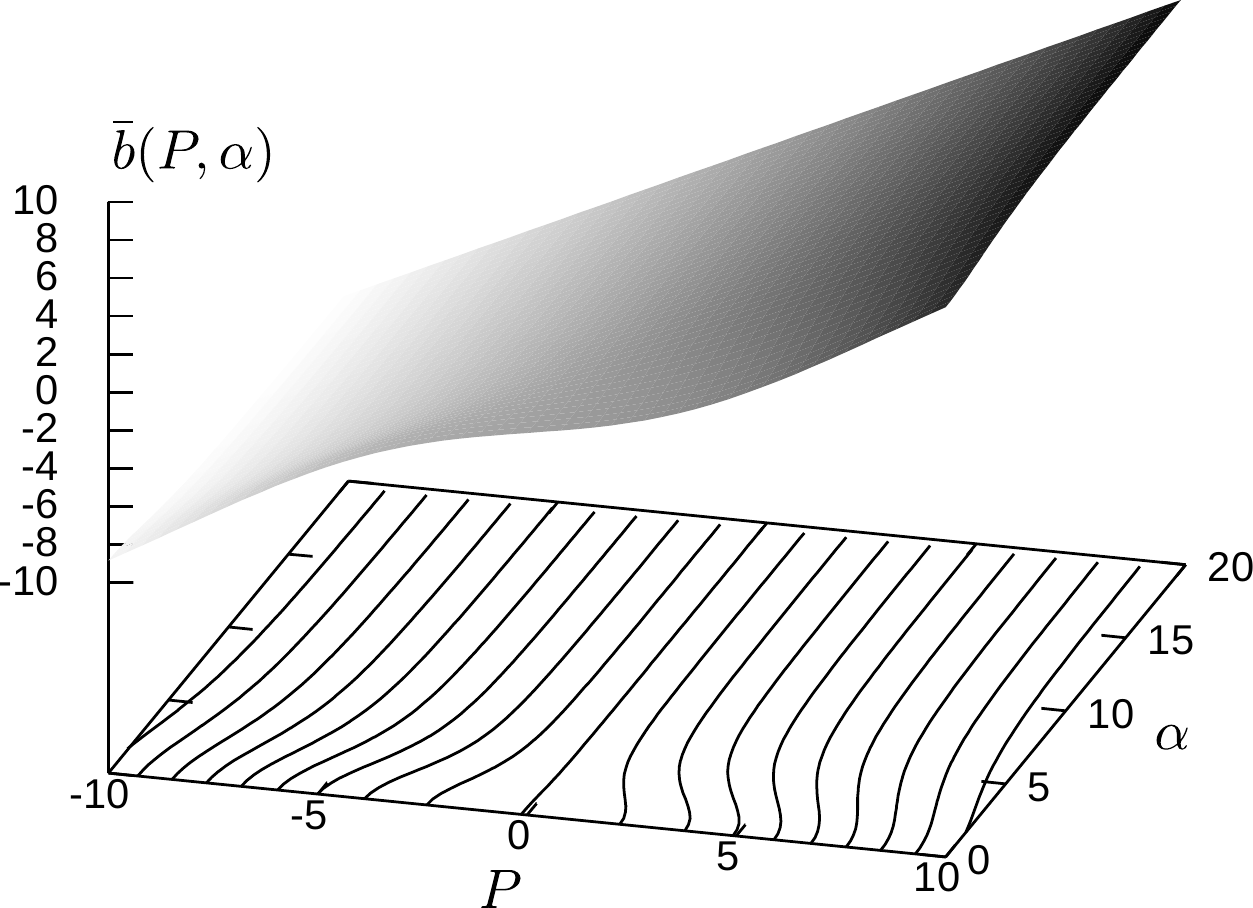} \\
(a)&(b)
\end{tabular}
\end{center}
\caption{Surfaces and level sets of $\bar H$ (a) and $\bar b$ (b) for $q=2$.}\label{HBQ2}
\end{figure}
\noindent In both cases $q=1,2$ we observe that $\bar H(P,\alpha)$ is ``almost'' quadratic in $P$ and nonincreasing in $\alpha$, while $\bar b(P,\alpha)$ exhibits a cubic-like 
behavior for $P$ close to zero. Let us take a closer look to some slices of these effective surfaces, in order to analyze their behavior for $P$ large. 
We choose $\alpha=10$ and we plot in Figure \ref{HBinf} the graphs of $\frac{\bar H(P,\alpha)}{P^2}$ and $\frac{|\bar b(P,\alpha)-P|}{P}$ as functions of $P$ for $q=1,2$.
We clearly observe convergence to $\frac{1}{2}$ and $0$ respectively, as expected by Proposition \ref{prop1}-(iv). \\

\begin{figure}[h!]
 \begin{center}
 \begin{tabular}{cc}
\includegraphics[width=.45\textwidth]{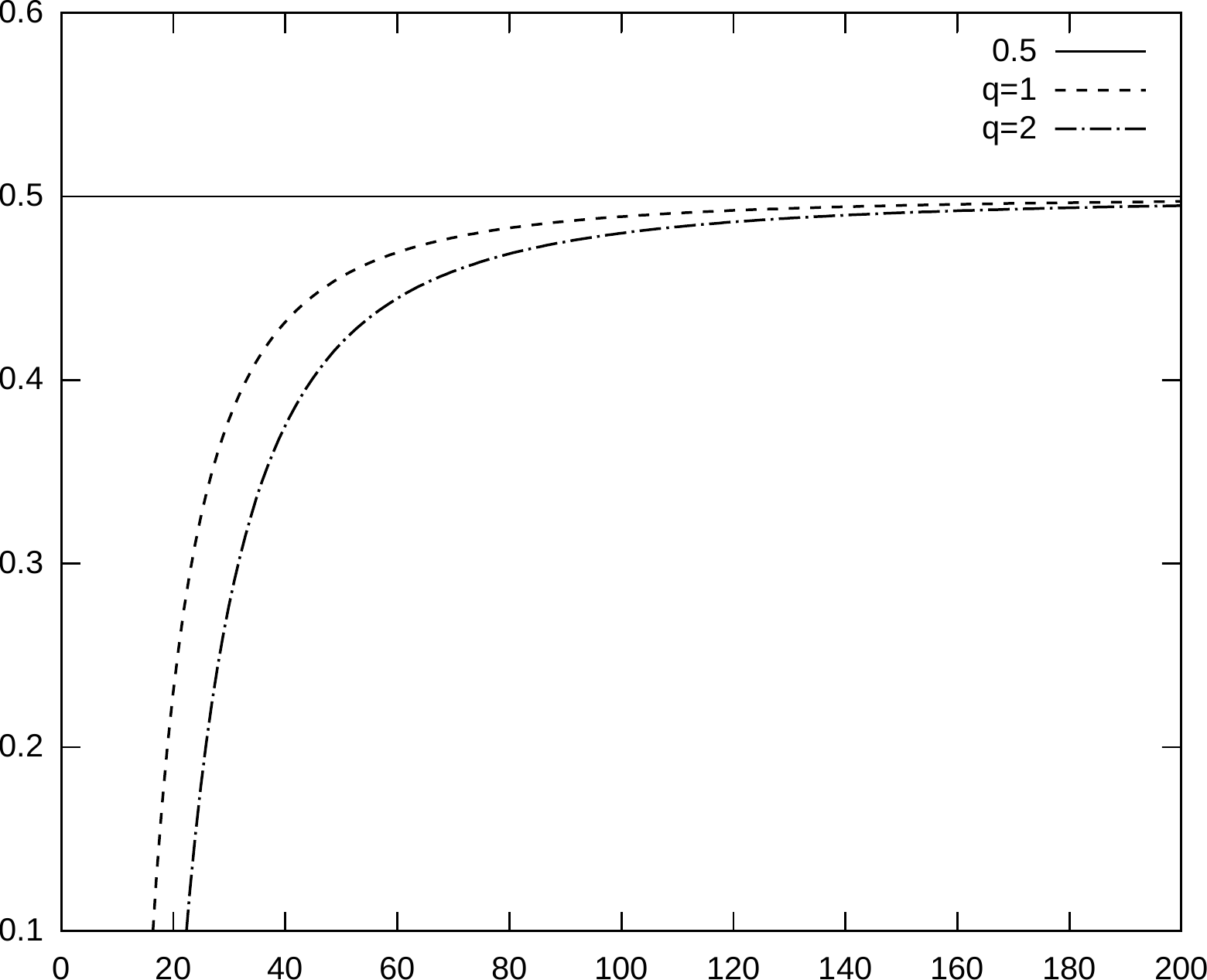} &
\includegraphics[width=.45\textwidth]{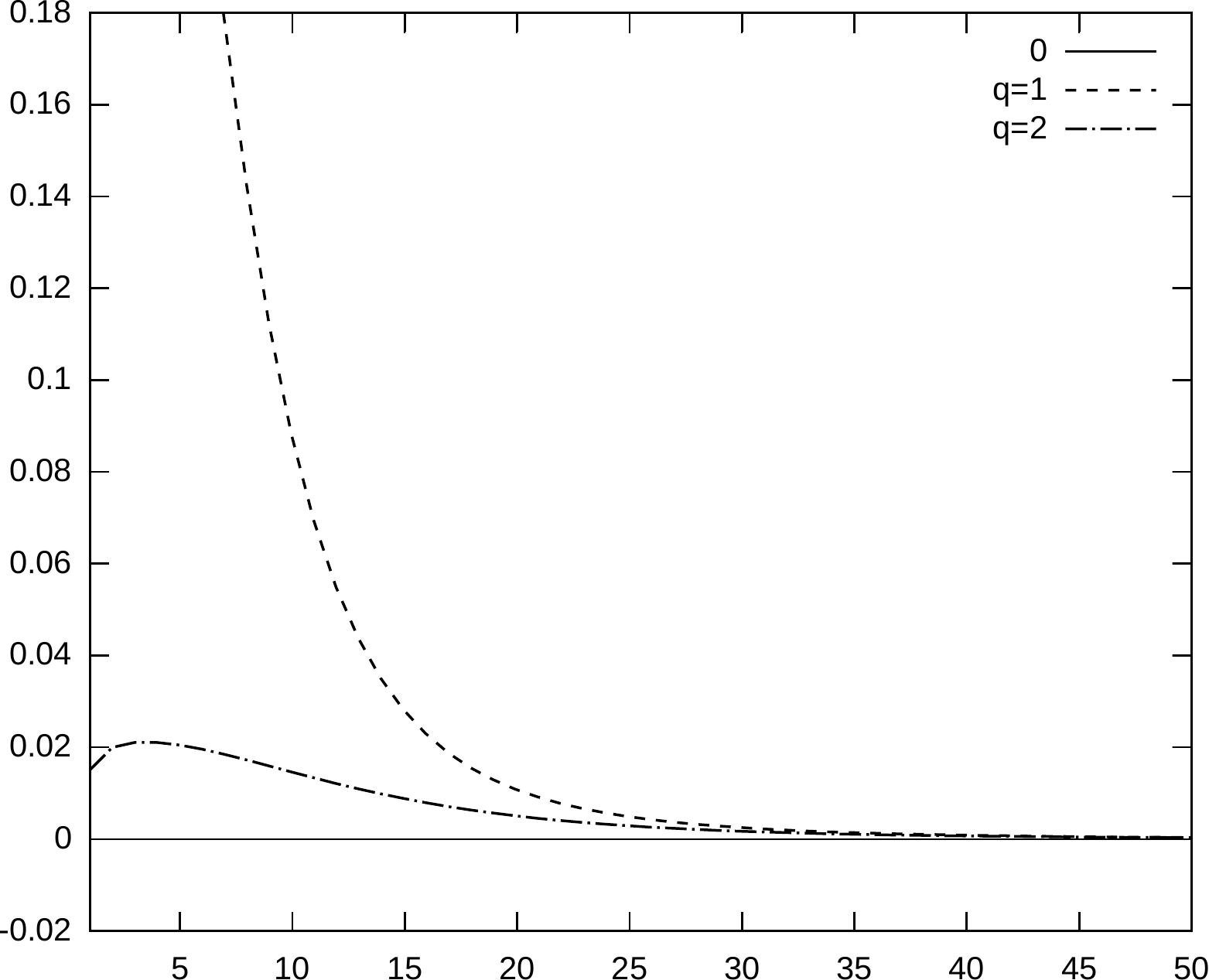} \\
(a)&(b)
\end{tabular}
\end{center}
\caption{Asymptotic behavior in $P$ of $\bar H(P,\alpha)/P^2$ (a) and $|\bar b(P,\alpha)-P|/P$ (b).}\label{HBinf}
\end{figure}

We proceed by examining the relationship between the effective operators described in the previous section.
First of all we want to show that, in the case of power nonlinearities, the structure of mean field game is lost during the homogenization process, namely we have 
$\nabla_P\bar H(P,\alpha)\neq \bar b(P,\alpha)$. To this end, we consider the residual in \eqref{H_P} as a measure of this discrepancy:
$$
\mathcal{R}(P,\alpha)=\left|\alpha \int_{\T^n}V_m(y,\alpha m) \tilde m m dy\right|\,,
$$
where we recall that, for each $(P,\alpha)$, the auxiliary $\tilde m$ solves the linearized cell problem \eqref{mvarP}, in which the previously computed 
correctors $(u,m)$ appear as given data. 

\noindent Figure \ref{QRp} shows the residuals respectively for $q=1$ and $q=2$ on the space of parameters 
$[-10,10]\times[0,20]$, while Figure \ref{QRp-inf} shows the behavior of the corresponding $L^\infty$ norms under grid refinement. 
We clearly observe that, as $N$ increases, the residual converges but not vanishes. 
Note that, in the case of power nonlinearities, the residual can be computed using also expression \eqref{nablaP}, in which appears 
the derivative in $P$ of the $q+1$ norm of the corrector $m$. Indeed, denoting by $m_{(P,\alpha)}$ the solution of \eqref{s2}, 
we can approximate $\frac{\partial}{\partial P}\| m_{(P,\alpha)}\|_{L^{q+1}}^{q+1}$ by finite differences of the form  
$\frac{\|m_{(P+\delta,\alpha)}\|_{L^{q+1}}^{q+1}- \|m_{(P,\alpha)}\|_{L^{q+1}}^{q+1}}{\delta}$ for some small $\delta>0$. This approach leads to similar results, 
up to an additional error of order $\mathcal{O}(\delta)$. \\

\begin{figure}[h!]
 \begin{center}
 \begin{tabular}{cc}
\includegraphics[width=.45\textwidth]{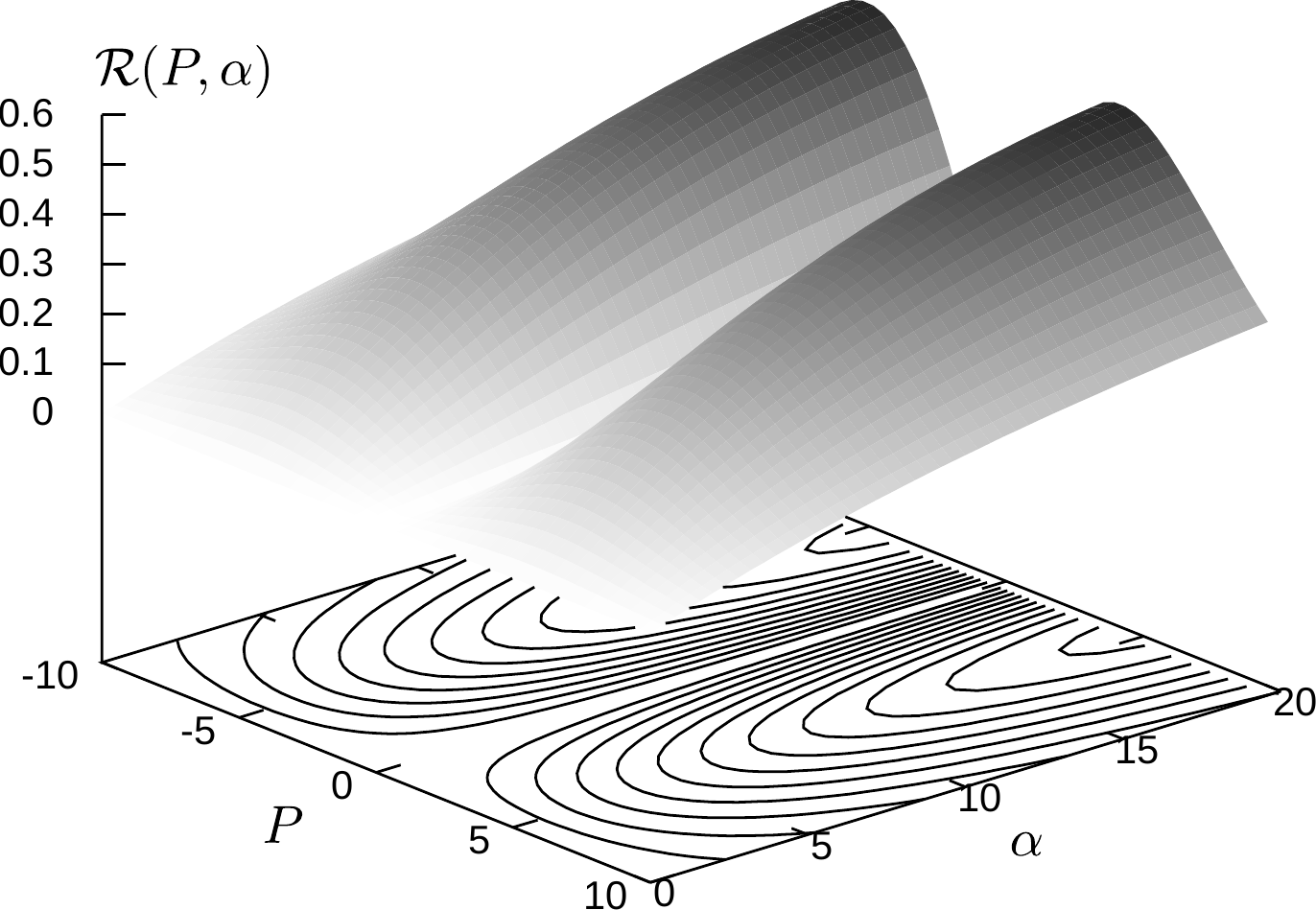} &
\includegraphics[width=.45\textwidth]{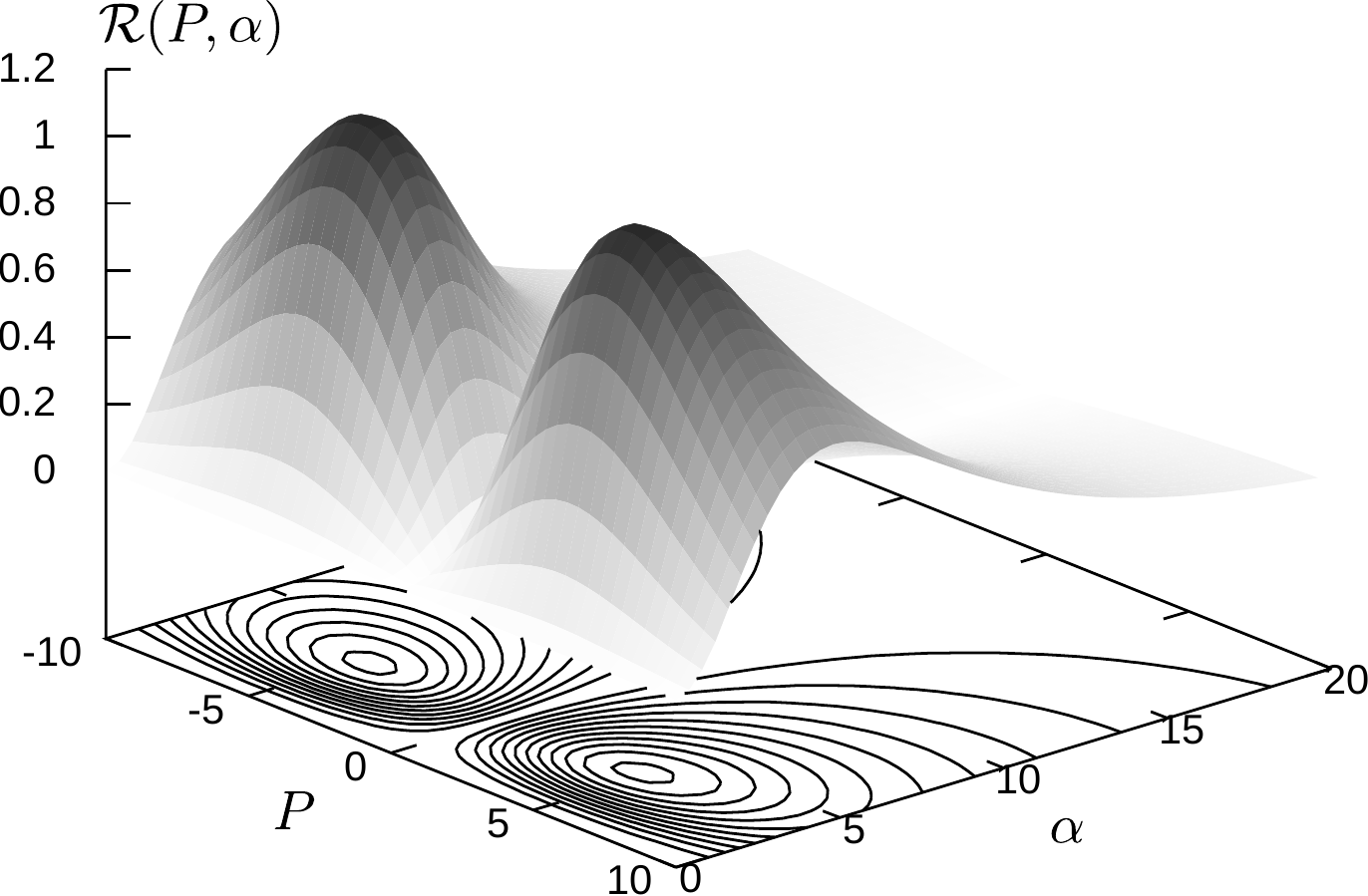} \\
(a)&(b)
\end{tabular}
\end{center}
\caption{Surface and level sets of the residual $\mathcal R$ for $q=1$ (a) and $q=2$ (b).}\label{QRp}
\end{figure}

\begin{figure}[h!]
 \begin{center}
 \begin{tabular}{cc}
\includegraphics[width=.45\textwidth]{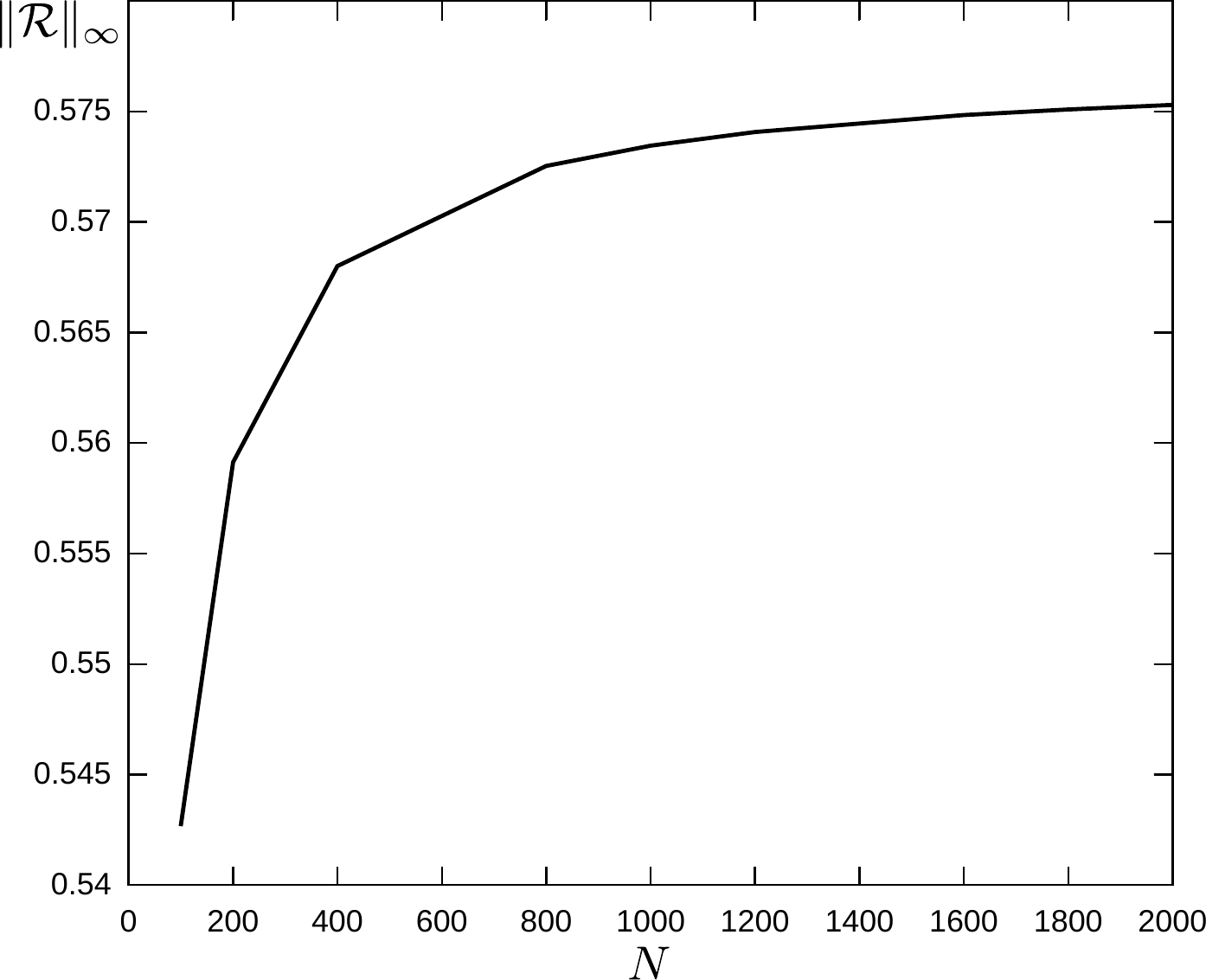} &
\includegraphics[width=.45\textwidth]{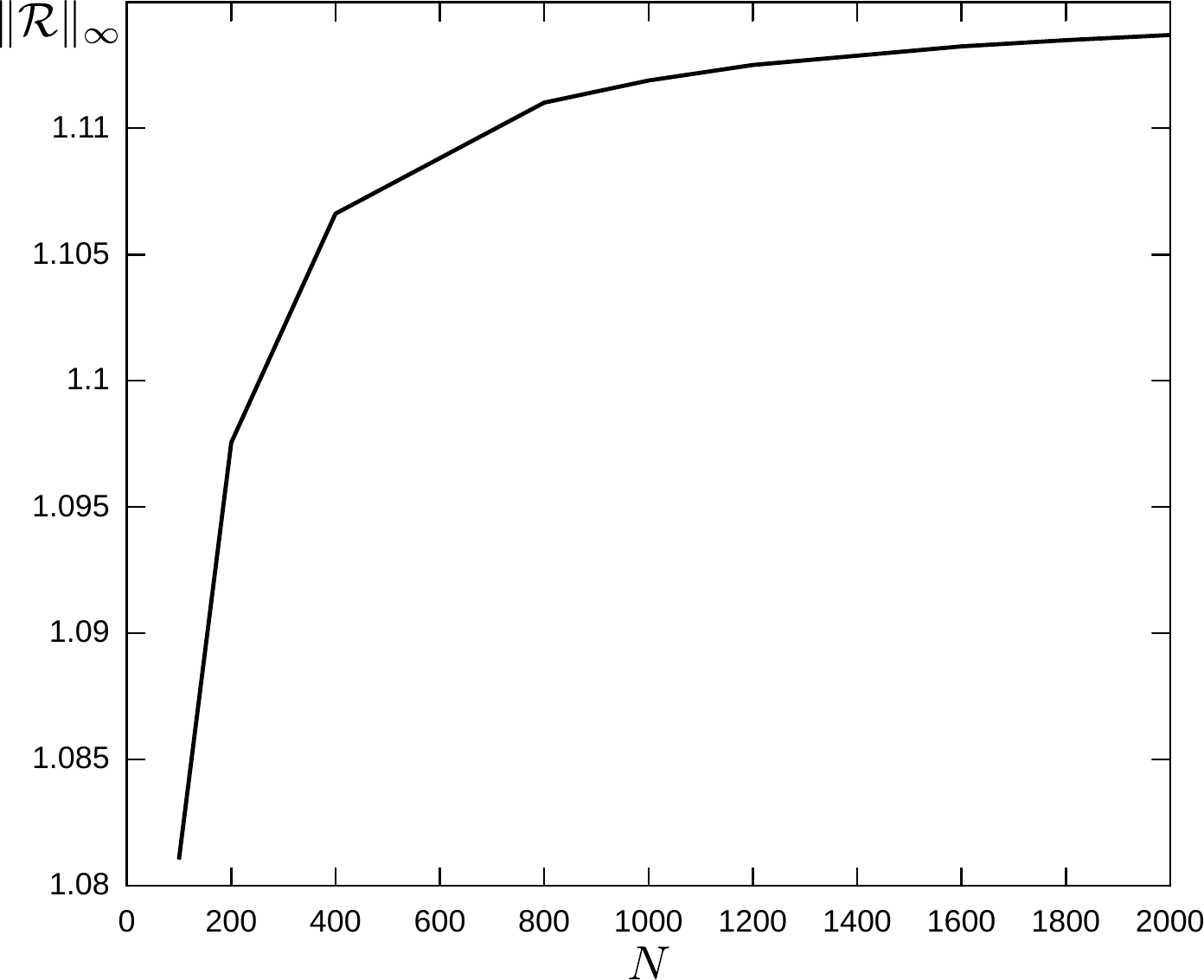} \\
(a)&(b)
\end{tabular}
\end{center}
\caption{$L^\infty$ norm of $\mathcal R$ under grid refinement for $q=1$ (a) and $q=2$ (b).}\label{QRp-inf}
\end{figure}

\noindent We evaluate the numerical error in the relationship \eqref{H_P} by defining 

$$\mathcal{E}_\infty=
\left\|\frac{\partial H}{\partial P}(P,\alpha) - \bar b(P,\alpha)+\alpha \int_{\T^n}V_m(y,\alpha m) \tilde m m dy\right\|_\infty\,,$$
Figure \ref{EpaNQ} shows that, for both $q=1$ and $q=2$, the error $\mathcal{E}_\infty$ has order of convergence $1$, under grid refinement for a space step $h=1/N$. 
\begin{figure}[h!]
 \begin{center}
 \begin{tabular}{cc}
\includegraphics[width=.45\textwidth]{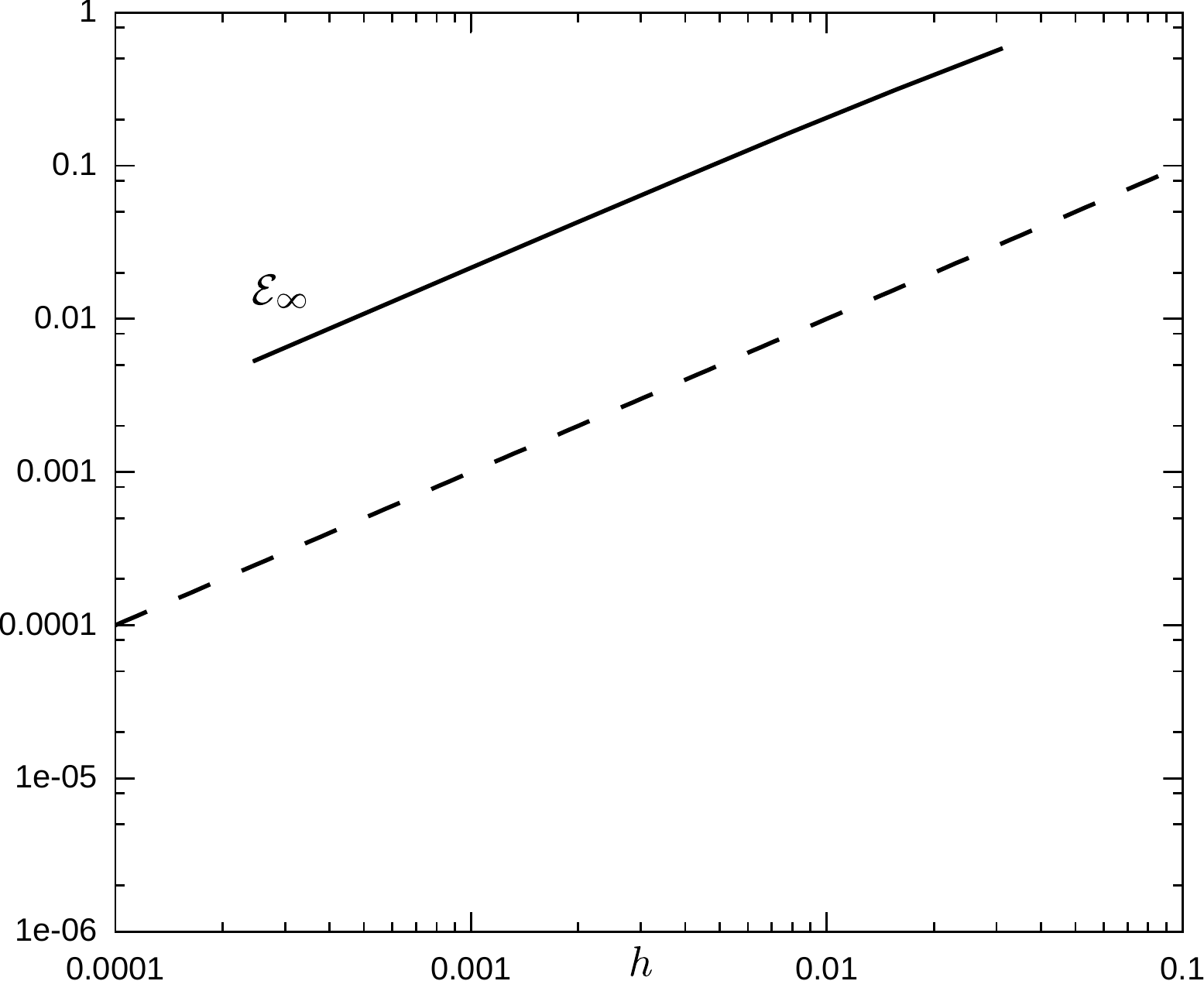} &
\includegraphics[width=.45\textwidth]{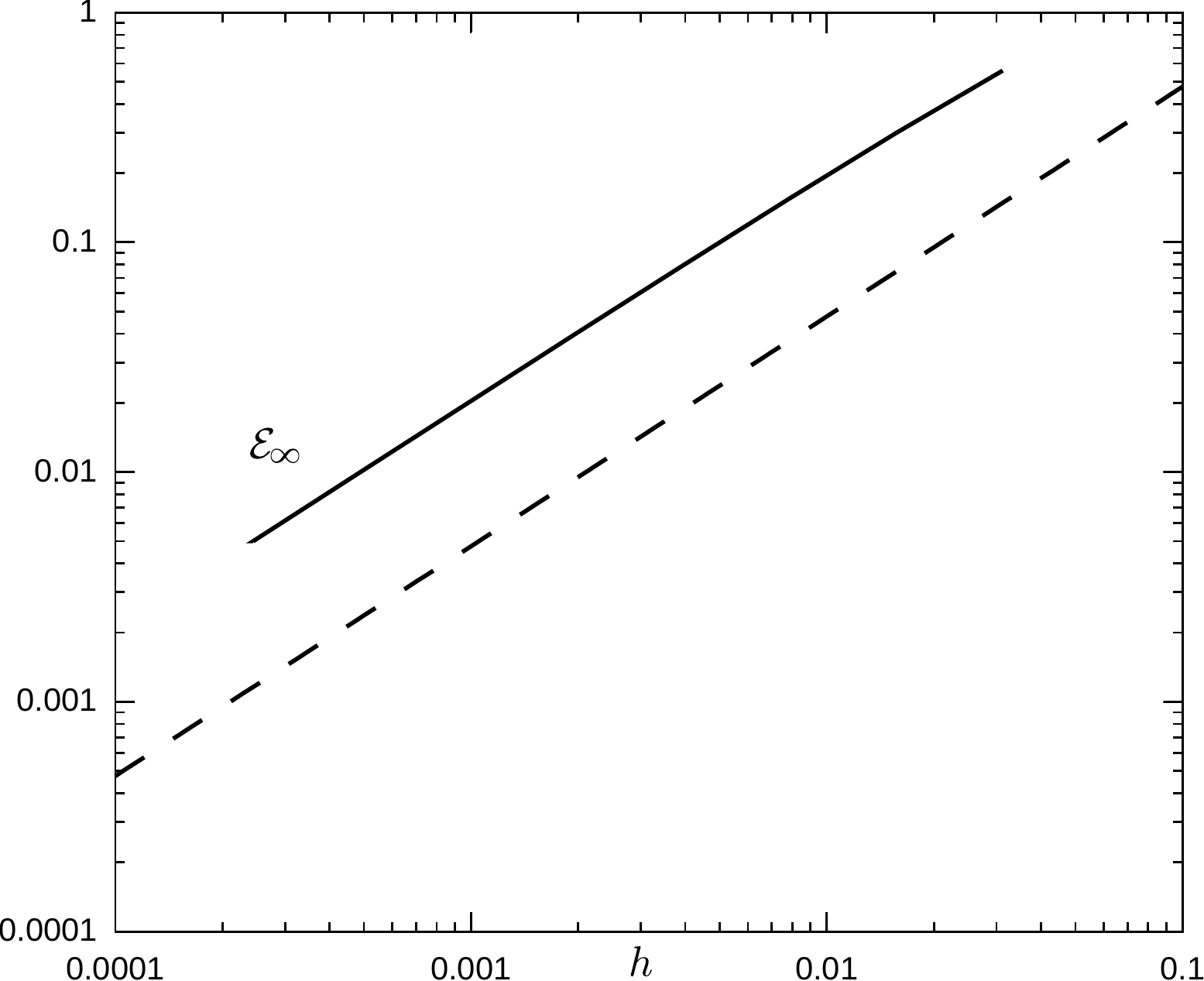} \\
(a) & (b)
\end{tabular}
\end{center}
\caption{$\mathcal{E}_\infty$ under grid refinement for $q=1$ (a) and $q=2$ (b) compared with straight (dashed) lines of slope $1$.}\label{EpaNQ}
\end{figure}\\

\noindent We now consider the logarithmic case, setting $V(y,m)=v(y)+\log m$. We already know that in this case the dependency of $\bar H$ on $P$ and $\alpha$ is separated, 
however we solve the full cell problem \eqref{ucorintro} and not the reduced one in \eqref{s3}, in order to recover this feature numerically.    
Again, we discretize the torus $\T^1$ with $N=400$ nodes and we consider a uniform grid of 
$51\times 51$ nodes for discretizing the space of parameters $[-10,10]\times[1,20]\ni (P,\alpha)$ (here we avoid the singularity of the $\log$ function just starting with $\alpha>0$). \\
Figure \ref{HBLOG} shows the surfaces and the level sets of the computed effective Hamiltonian $\bar H(P,\alpha)$ and the effective drift $\bar b(P,\alpha)$ as functions of $(P,\alpha)$. 
We observe in particular the convexity of $\bar H$ in $P$ and the independence of $\bar b$ on $\alpha$. 
We finally recall that in this case the residual $\mathcal{R}(P,\alpha)=\left|\alpha \int_{\T^n}V_m(y,\alpha m) \tilde m m dy\right|=0$, being $V_m=1/m$ and $\tilde m$ with zero mean. 
Then, in Figure \ref{LOGR-inf} we show the error 
$\mathcal E(P,\alpha)=|\nabla_P\bar H(P,\alpha)-\bar b(P,\alpha)|$ and the behavior of its $L^\infty$ norm $\mathcal{E}_\infty$ under grid refinement for a space step $h=1/N$. 
We see that $\mathcal {E}_\infty$ does not depend 
on $\alpha$ and it exihibits again an order of convergence $1$ as the space step goes to zero. This confirm the relation $\nabla_P\bar H(P,\alpha)=\bar b(P,\alpha)$ in the logarithmic case and also the separated dependency 
of $\bar H$ on $P$ and $\alpha$.\\
\begin{figure}[h!]
 \begin{center}
 \begin{tabular}{cc}
\includegraphics[width=.45\textwidth]{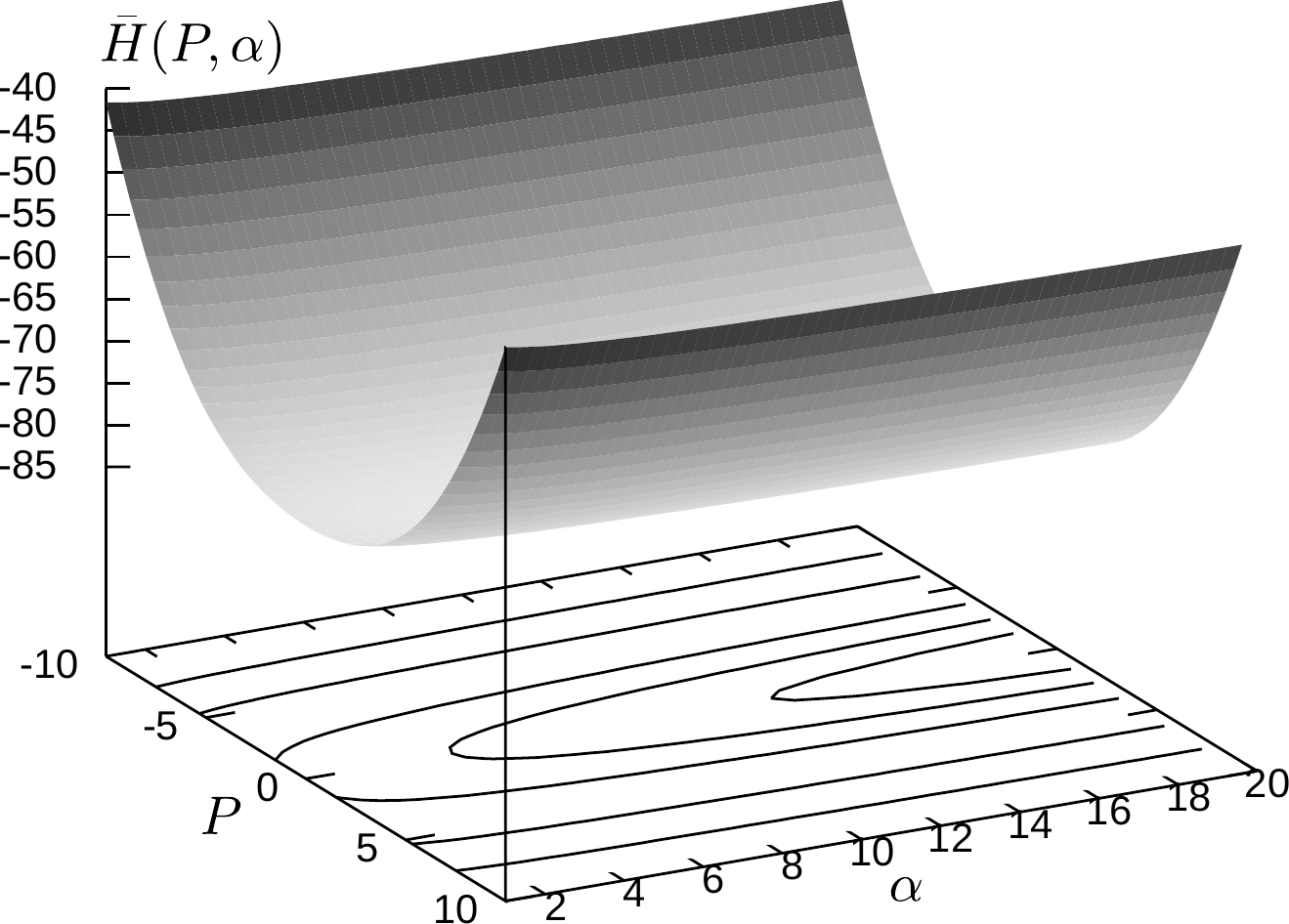} &
\includegraphics[width=.45\textwidth]{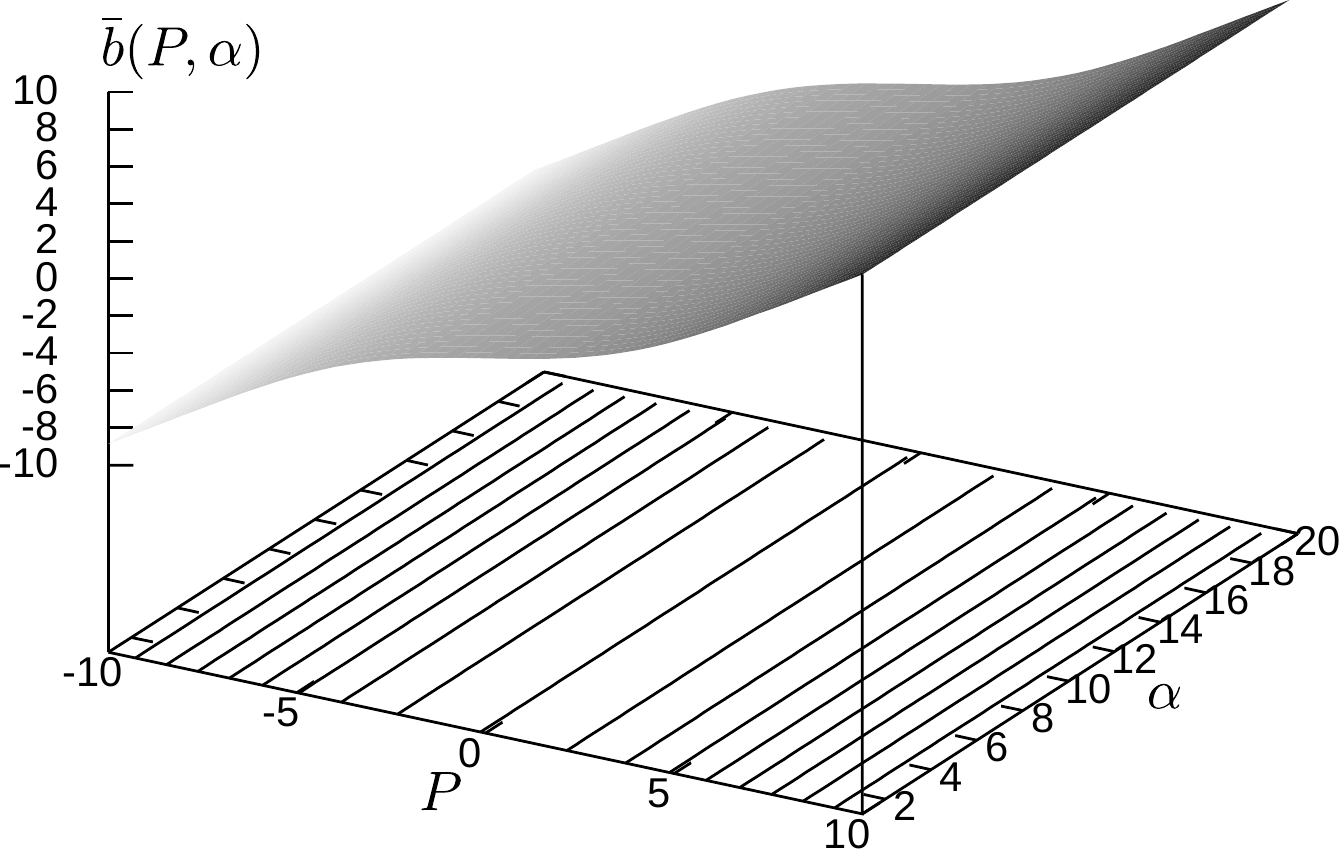} \\
(a)&(b)
\end{tabular}
\end{center}
\caption{Surfaces and level sets of $\bar H$ (a) and $\bar b$ (b).}\label{HBLOG}
\end{figure}
\begin{figure}[h!]
 \begin{center}
 \begin{tabular}{cc}
\includegraphics[width=.5\textwidth]{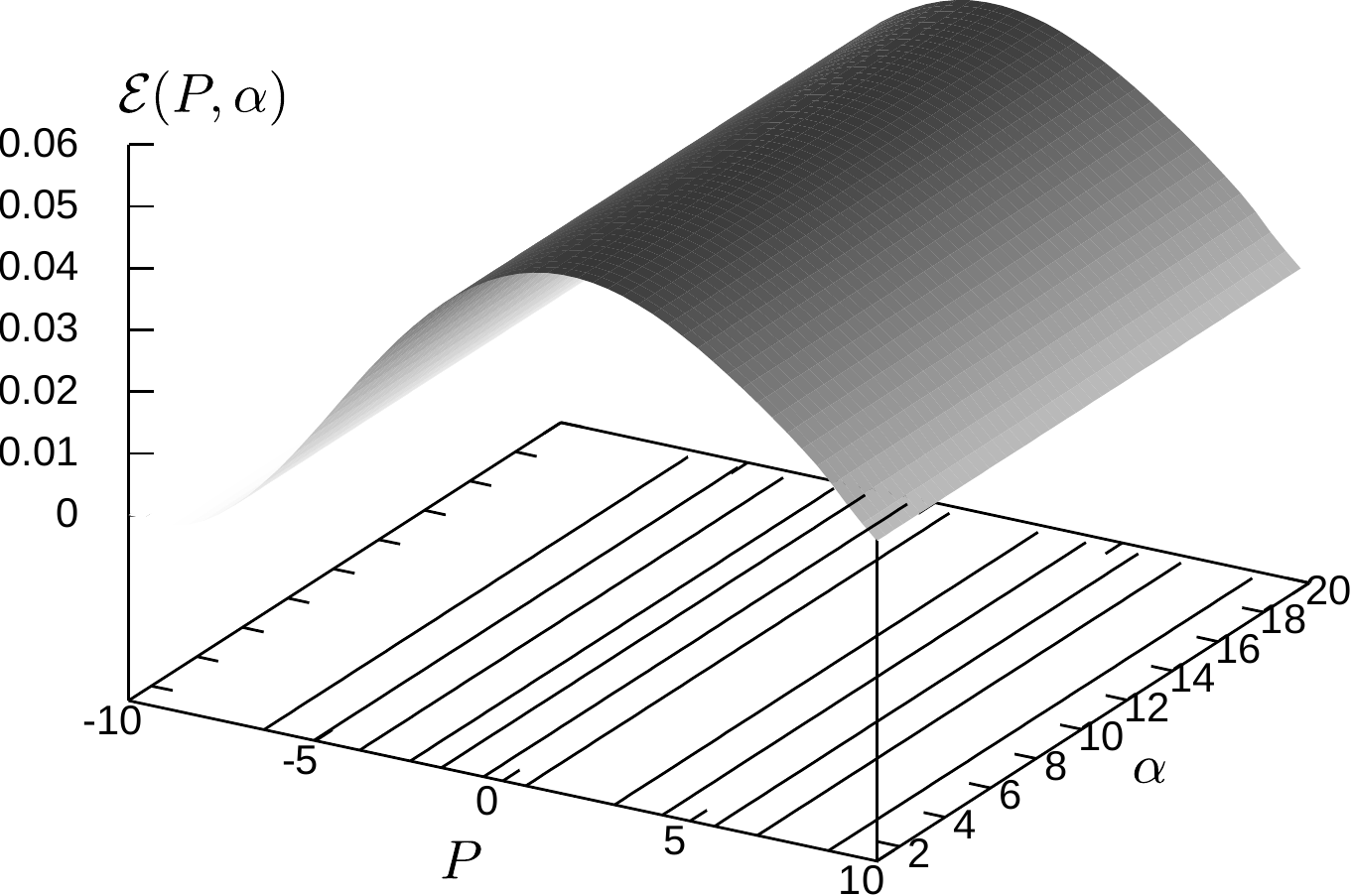} &
\includegraphics[width=.45\textwidth]{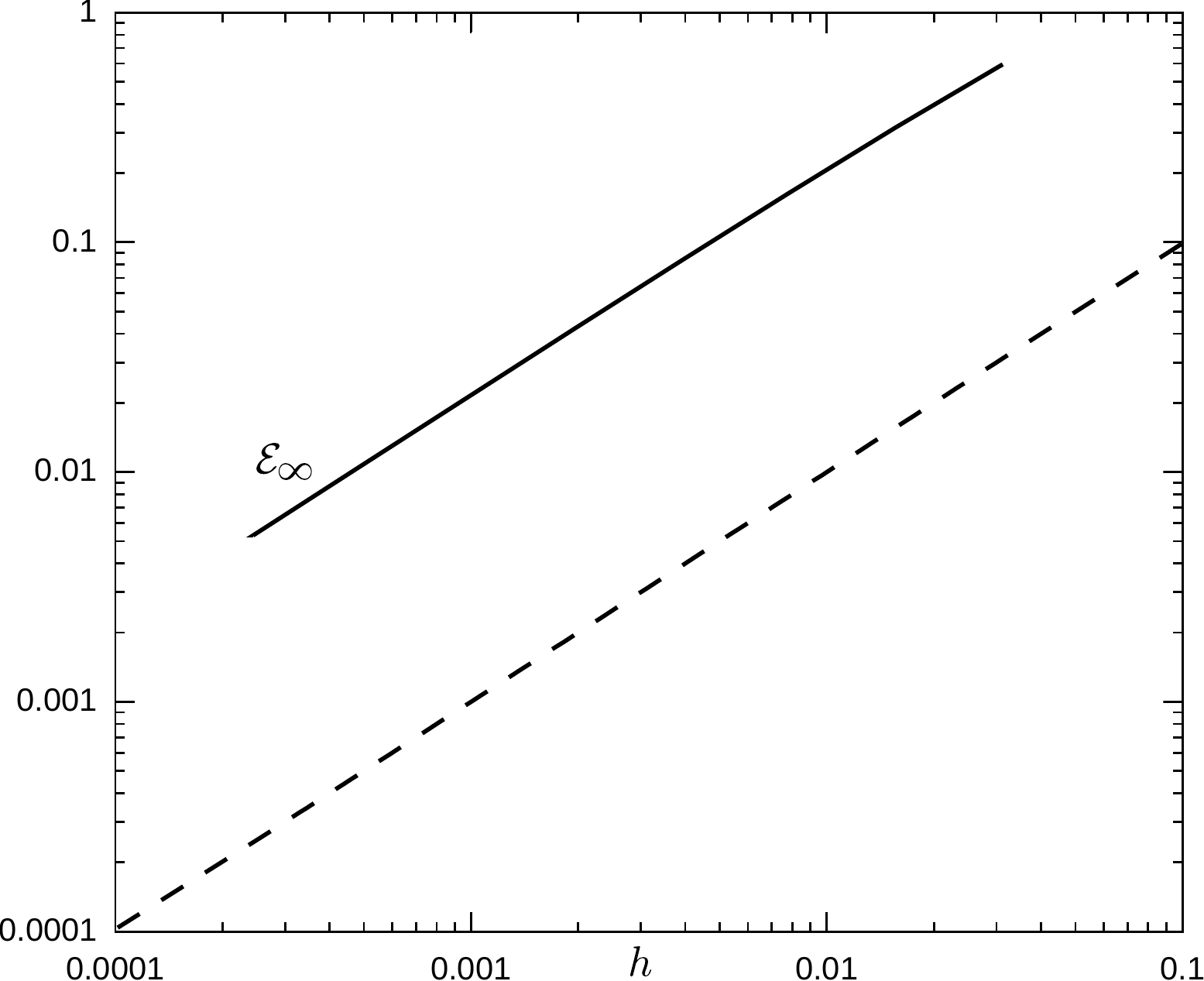} \\
(a)&(b)
\end{tabular}
\end{center}
\caption{The error $\mathcal E(P,\alpha)=|\nabla_P\bar H(P,\alpha)-\bar b(P,\alpha)|$ (a), and its $L^\infty$ norm under grid refinement compared with a straight (dashed) line of slope $1$ (b).}\label{LOGR-inf}
\end{figure}

\end{document}